\documentclass[sn-mathphys,Numbered]{sn-jnl}

\usepackage{lineno,hyperref}
\usepackage{booktabs}
\usepackage{multirow}
\usepackage{amsthm}
\usepackage{graphicx}
\usepackage{subfigure}
\usepackage{float}
\usepackage{bm}
\usepackage{lipsum}
\usepackage[marginal]{
	footmisc}
\usepackage{listings} 
\usepackage{color}
\usepackage{amsfonts,enumerate,amsmath,amssymb}
\newtheorem{thm}{Theorem}[section]

\newtheorem{prop}[thm]{Proposition}
\newtheorem{lem}[thm]{Lemma}

\newtheorem{rem}[thm]{Remark}

\newtheorem{exm}[thm]{Example}

\newcommand{\thmref}[1]{Theorem~{\rm \ref{#1}}}
\newcommand{\lemref}[1]{Lemma~{\rm \ref{#1}}}
\newcommand{\propref}[1]{Proposition~{\rm \ref{#1}}}

\def\para#1{\vskip .4\baselineskip\noindent{\bf #1}}
\numberwithin{equation}{section}
\usepackage{graphicx}%
\usepackage{multirow}%
\usepackage{amsmath,amssymb,amsfonts}%
\usepackage{amsthm}%
\usepackage{mathrsfs}%
\usepackage[title]{appendix}%
\usepackage{xcolor}%
\usepackage{textcomp}%
\usepackage{manyfoot}%
\usepackage{booktabs}%
\usepackage{algorithm}%
\usepackage{algorithmicx}%
\usepackage{algpseudocode}%
\usepackage{listings}%

\theoremstyle{thmstyleone}%
\theoremstyle{thmstyletwo}%

\theoremstyle{thmstylethree}%

\begin{document}

\title[Moderate deviations for two-time scale systems with mixed fractional Brownian motion]{Moderate deviations for two-time scale systems  with mixed fractional Brownian motion}


\author[1]{\fnm{Xiaoyu} \sur{\textsc{{Yang}}}}\email{yangxiaoyu@yahoo.com}

\author[1]{\fnm{Yuzuru} \sur{\textsc{{Inahama}}}}\email{inahama@math.kyushu-u.ac.jp}

\author*[2]{\fnm{Yong} \sur{\textsc{{Xu}}}}\email{hsux3@nwpu.edu.cn}

\affil[1]{\orgdiv{Faculty of Mathematics}, \orgname{Kyushu University}, \orgaddress{\city{Fukuoka}, \postcode{8198395}, \country{Japan}}}

\affil*[2]{\orgdiv{School of Mathematics and Statistics}, \orgname{Northwestern Polytechnical University}, \orgaddress{\city{Xi'an}, \postcode{710072}, \country{China}}}


\abstract{This work focuses on  moderate deviations for two-time scale systems with mixed fractional Brownian motion. Our proof uses the weak convergence method which is based on the  {variational representation formula for  mixed fractional Brownian motion}.  Throughout this paper,  Hurst parameter of fractional Brownian motion is larger than $1/2$ and the integral along  the fractional Brownian motion is understood as the generalized Riemann-Stieltjes integral. First, we consider  single-time scale systems with fractional Brownian motion. 	The key of our proof is showing the weak convergence of the controlled system. Next, we  extend our method to show moderate deviations for  two-time scale systems.  To this goal, we combine the Khasminskii-type averaging principle and  the weak convergence approach. 	}




\maketitle

	\section{Introduction}\label{sec-1}
In this work, we consider the following stochastic differential equation (SDE) driven by fractional Brownian motion (FBM) in $\mathbb{R}^{m}$ on the time interval $[0,T]$:
\begin{eqnarray}\label{0}
dx^{\varepsilon}_t = f(x^{\varepsilon}_t)dt + \sqrt \varepsilon  \sigma( x^{\varepsilon}_t)dB^H_{t},\quad x^{\varepsilon}_0=x_0\in \mathbb{R}^{m},
\end{eqnarray}
where $\varepsilon\in(0,1]$
is a small parameter, $B^H =(B^H_t)_{t \in [0,T]}\in \mathbb{R}^{d_1}$ is FBM with $H\in(1/2,1)$,
$f\colon \mathbb{R}^{m}\to \mathbb{R}^{m} $ and	 $\sigma\colon\mathbb{R}^{m}\to \mathbb{R}^{m\times d_1}$  are  nonlinear functions satisfying some suitable conditions which will be given in Section 3. Then there exists a unique (pathwise) solution $x^{\varepsilon}\in C^{\alpha}\left([0,T], \mathbb{R}^m\right)$ to the SDE (\ref{0})  with $\alpha\in (1-H,1/2)$ where $C^{\alpha}\left([0,T], \mathbb{R}^m\right)$ is the $\alpha$-H\"older continuous 
path space.

As $\varepsilon \to 0$, it is easy to see that  $x^\varepsilon$  {converges to the solution
	of the following deterministic ordinary differential equation (ODE)}:
\begin{eqnarray}\label{01}
dx_t = f(x_t)dt ,\quad x_0\in \mathbb{R}^{m}.
\end{eqnarray}
Then, it is important to study the fluctuation of $x^\varepsilon$  around  $x$. To do this we define the deviation component between $x^\varepsilon$ and $x$ by 
\begin{eqnarray}\label{02}
z^\varepsilon_t = \frac{x^\varepsilon_t -x_t }{\sqrt{\varepsilon}h(\varepsilon)},\quad z^\varepsilon_0=0.
\end{eqnarray}
Here,  { $h:(0,1]\to (0,\infty)$ is continuous and satisfies that $\lim_{\varepsilon \to  0} h(\varepsilon)= \infty$ and $\lim_{\varepsilon \to  0}\sqrt{\varepsilon}h(\varepsilon)=0$ for all $\varepsilon\in (0,1]$.}	(The case $h(\varepsilon)\equiv1$ corresponds to the central limit theorem (CLT), while the case $h(\varepsilon)=\frac{1}{\sqrt{\varepsilon}}$ corresponds to the  large deviation principle (LDP). Therefore,   our intermediate case $\lim_{\varepsilon \to  0}\sqrt{\varepsilon}h(\varepsilon)=0$ is called moderate deviation principle (MDP).)
 {A simple example is $h(\varepsilon)=\varepsilon^{-\frac{\theta}{2}}$ with $\theta\in(0,1)$. }
 {The MDP can be considered a bridge between the CLT and the LDP.}

The first main purpose of  this work is to prove an MDP for  $x^{\varepsilon}$ of the above system \eqref{0}.  
The family
$\{x^{\varepsilon}\}_{\varepsilon\in(0,1]}\subset C^{\alpha}\left([0,T], \mathbb{R}^m\right)$  is said to satisfy an MDP on $C^{\alpha}\left([0,T], \mathbb{R}^m\right)$ with speed $b(\varepsilon)=1/h^2(\varepsilon)$ and a
rate function $I: C^{\alpha}\left([0,T], \mathbb{R}^m\right)\rightarrow [0, \infty]$ ($1-H <\alpha <1/2$) if the following two conditions hold:
 {\begin{itemize}
	\item For each closed subset $F$ of $C^{\alpha}\left([0,T], \mathbb{R}^m\right)$,
	$$\limsup _{\varepsilon \rightarrow 0} b(\varepsilon) \log \mathbb{P}\big(z^{\varepsilon} \in F\big) \leqslant-\inf _{z \in F} I(z).$$
	\item For each open subset $G$ of $C^{\alpha}\left([0,T], \mathbb{R}^m\right)$,
	$$\liminf _{\varepsilon \rightarrow 0} b(\varepsilon)  \log \mathbb{P}\big(z^{\varepsilon} \in G\big) \geqslant-\inf _{z \in G} I(z).$$
\end{itemize}}
The above MDP will be proved in Theorem \ref{thm1}, in which $I$ is shown  to be a good rate function.

There are many references that explore the MDP. It can be traced back to  MDP for independently and identically distributed random variables, see \cite{Amosova1974, Amosova1979,Chen1991,Chen1997}. Subsequently, some results on MDP for empirical processes in finite-dimensional space have been given \cite{Arcones2003}.
For the continuous-time stochastic processes, there were also some results. Based on the variational representation for Brownian motion (BM) and Poisson random measure \cite{BP_book,1998Dupuis,2011Variational}, along the weak convergence approach, Budhiraja, Dupuis and Ganguly  established the MDP for stochastic dynamical system with jumps \cite{Budhiraja2016}. Then, Suo, Tao and Zhang focused on the  stochastic delay system with BM under polynomial growth condition \cite{Suo2018}. However, it is an open problem for non-Markovian stochastic processes such as solutions of a stochastic dynamical system driven by FBM. As FBM is neither  a Markov process nor  a martingale,     {it is used to capture long-range dependence} in real complex systems \cite{2008Biagini}. Its Hurst parameter $H$  describes the roughness of  sample paths.  {The smaller $H$ is, the rougher sample paths of FBM become}. Therefore,  we will first study   the MDP  for such a stochastic dynamical system driven by FBM.

Nowadays, the two-time scale system,  composed of two subsystems with different time scales, has been widely applied in various fields, such as financial economics \cite{2001BNS}, climate-weather (see  \cite{2000Kiefer}), biological field and so on \cite{Krupa2008Mixed}. 
So  we consider the following 
two-time scale SDEs on the time interval $[0,T]$
driven by BM and  FBM in $\mathbb{R}^{m}\times \mathbb{R}^{n}$:
\begin{eqnarray}\label{1}
\left
\{
\begin{array}{ll}
dx^{\varepsilon}_t = f_{1}(x^{\varepsilon}_t, y^{\varepsilon}_t)dt + \sqrt \varepsilon  \sigma_{1}( x^{\varepsilon}_t)dB^H_{t},\\
dy^{\varepsilon}_t = \frac{1}{\varepsilon}f_{2}( x^{\varepsilon}_t, y^{\varepsilon}_t)dt +\frac{1}{\sqrt \varepsilon} \sigma_2( x^{\varepsilon}_t, y^{\varepsilon}_t)dW_{t}
\end{array}
\right.
\end{eqnarray}
with a (non-random) initial condition 
$(x^{\varepsilon}_0, y^{\varepsilon}_0)=(x_0, y_0)\in \mathbb{R}^{m}\times \mathbb{R}^{n}$.
Here, 
$f_1\colon \mathbb{R}^{m}\times \mathbb{R}^{n} \to  \mathbb{R}^{m}$, $f_2\colon  \mathbb{R}^{m}\times \mathbb{R}^{n}\to \mathbb{R}^{n}$, 	 $\sigma_{1}\colon \mathbb{R}^{m}\to \mathbb{R}^{m\times d_1}$ and 
 {$\sigma_{2}\colon \mathbb{R}^{m}\times \mathbb{R}^{n}\to \mathbb{R}^{m\times d_2}$} are suitable nonlinear functions. 
Precise conditions on these coefficient functions will be specified later. Let
$W =(W_t)_{t \in [0,T]}$ be  $d_2$-dimensional standard Brownian motion (BM) independent of FBM $B^H$. Then $(B^H,W)$ is called   $d$-dimensional mixed FBM 
with Hurst parameter $H \in (1/2, 1)$.  
The small parameter $0< \varepsilon \ll 1$ are used to describe the time scale separation between the slow component $x^{\varepsilon}$ and the fast component $y^{\varepsilon}$.	Section 4 will present  appropriate conditions under which the two-time scale system (\ref{1}) possesses a unique (pathwise) solution $(x^\varepsilon, y^\varepsilon) \in C^\alpha([0,T], \mathbb{R}^m) \times C([0,T], \mathbb{R}^n)$ with $\alpha \in (1-H, 1/2)$. Here, 
$C\left([0,T], \mathbb{R}^n\right)$ is the  continuous path space with the uniform topology. The subsequent section will provide a comprehensive explanation of this solution.


According to the averaging principle \cite{2021pei,2023Pei,2015Xu}, we can see that under appropriate conditions, the slow component in \eqref{1} will strongly converge to the solution of the following ODE in $C^\alpha([0,T], \mathbb{R}^m)$ as $\varepsilon \to 0$, which is derived from \cite{2023Pei},  
\begin{equation}\label{4}
\left\{\begin{array}{l}
d\bar{x}_t=\bar{f}_1\left(\bar{x}_t\right) d t \\
\bar{x}_0=x_0 \in \mathbb{R}^m,
\end{array}\right.
\end{equation}
where $\bar{f}_1(x)=\int_{\mathbb{R}^{n}} f_1(x, y) d \mu_x(y)$ and $\mu_x$ is a unique invariant probability measure of the fast component with the ``frozen" slow variable  $x \in \mathbb{R}^m$.

For  moderate deviations in two-time scale system case,  Guillin and Liptser gave the MDP for two-time scale diffusions in finite dimensional space \cite{Guillin2005}. Then, Chigansky and Liptser constructed the MDP for diffusion process in a random environment \cite{Chigansky2010}. It is worth mentioning that Gasteratos considered the two-time scale reaction-diffusion system under FBM using weak convergence and viable pair method \cite{Gasteratos2023}, which is different from our work. Based on the weak convergence and averaging principle,  de Oliveira Gomes and Catuogno  gave the MDP for two-time scale system with  {Poisson} jumps \cite{Gomes2023}.  {
	Different from the above works, where the topology is the uniform one on the path space, 
in our main results, however,
	the topology is the $\alpha$-H\"older topology ($1-H < \alpha <1/2$).
Besides, we need to state the differences between \cite{Bourguin} and this work except the topology issue. Firstly, for the two-time scale system in \cite{Bourguin}, the    fast system is independent of slow component, but in our system \eqref{1}, the fast component fully depends on the slow component.  Secondly, in \cite{Bourguin}, they applied   viable pair method, but in our work, we use weak convergence method and averaging principle. }

To sum up, we first target the MDP for a single time scale system under FBM. By definition,   the MDP for $x^\varepsilon$ is equivalent to  the LDP for $z^\varepsilon$. Next, according to the equivalence between the Laplace principle and the LDP, our problem is transformed into showing the Laplace upper and lower bounds  for $z^\varepsilon$. Based on the variational representation formula for a BM \cite{BP_book}, we then focus on the weak convergence of the controlled system. For this purpose, the Taylor expansion and some boundness results for the deviation component are used. We then extend the MDP result to the two-time scale system. Different from Gasteratos, Salins, Spiliopoulos' result \cite{Gasteratos2023}, we consider the two-time scale system where the fast component is fully dependent on the slow component and establish the result in the $\alpha$-H\"older topology space. Along with the weak convergence approach, we will apply the Khasminskii-type averaging principle, the above MDP  for the single-time scale system, and the time discretization technique to solve this problem.

The structure of this paper is as follows. In Section 2, we introduce the notations and  preliminaries. In Section 3, we introduce the single-time scale system with FBM, then give needed assumptions and a precise  statement of our MDP result (Theorem \ref{thm1}). 
In Section 4, we extend to the two-time scale system with mixed FBM. 

Throughout this paper, we set 
{$\mathbb{N}=\{1,2, \ldots\}$}
and denote by
$c$, $C$, $c_1$, $C_1$, etc. certain positive constants that may change from line to line. {The time horizon $T >0$ and the non-random
	initial value 
	$x_0, (x_0,y_0)$ are arbitrary but fixed throughout this paper. We do not keep track of dependence on 
	$H \in (1/2, 1)$, $\alpha \in (1-H, 1/2)$,
	$T$, $x_0, (x_0,y_0)$.} The norm of matrices is the Hilbert-Schmidt norm.

	\section{Preliminaries}\label{sec-2}
\subsection{Fractional integrals and Spaces}\label{sec-2-1}
Now we introduce fractional integrals, some  spaces and norms that will be used in this work. 	
In this subsection $k, l \in \mathbb{N}$ and $0<\alpha <1/2$.

For  measurable functions $g:[0,T]\to \mathbb{R}^{k}$, denote that $W_0^{\alpha, 1} = W_0^{\alpha, 1} ([0,T], \mathbb{R}^{k})$ with norm
$$
\|g\|_{\alpha, 1}:=\int_0^T \frac{|g(s)|}{s^\alpha} d s+\int_0^T \int_0^s \frac{|g(s)-g(y)|}{(s-y)^{\alpha+1}} d y d s<\infty.
$$	
Denote by $W_0^{\alpha,\infty} = W_0^{\alpha,\infty} ([0,T], \mathbb{R}^{k})$ 
the space of measurable functions $g:[0,T]\to \mathbb{R}^{k}$ such that
$$\|g\|_{\alpha, \infty}:=\sup _{t \in[0, T]}\|g\|_{\alpha,[0,t]}<\infty,	$$
where 
$$\|g\|_{\alpha,[0,t]}=|g(t)|+\int_0^t \frac{|g(t)-g(s)|}{(t-s)^{\alpha+1}} d s.$$
Next, we introduce the  {the space	of H\"older continuous paths}. For $\eta \in(0,1]$,  let $C^\eta = C^\eta  ([0,T], \mathbb{R}^{k})$ be the space of $\eta$-H\"older continuous functions $g:[0,T]\to \mathbb{R}^{k}$,  with the norm
 {$$
\|g\|_{\eta\textrm{-hld}}:=\|g\|_{\infty}+\sup _{0 \leq s<t \le T} \frac{|g(t)-g(s)|}{(t-s)^\eta}<\infty,
$$}
where $\|g\|_{\infty}=\sup _{t \in[0, T]}|g(t)|$.
For any $\kappa \in (0,\alpha)$, we have the following  continuous inclusion $C^{\alpha+\kappa} \subset W_0^{\alpha, \infty} \subset C^{\alpha-\kappa}$.

Denote that  $W_T^{1-\alpha, \infty} =W_T^{1-\alpha, \infty} ([0,T], \mathbb{R}^{k})$ is the space of measurable functions $h:[0,T]\to \mathbb{R}^{k}$ such that
$$
\|h\|_{1-\alpha, \infty, T}:=\sup _{0\le s<t\le T}\left(\frac{|h(t)-h(s)|}{(t-s)^{1-\alpha}}+\int_s^t \frac{|h(y)-h(s)|}{(y-s)^{2-\alpha}} d y\right)<\infty.
$$
It is also easy to see that  {in fact, for any $\kappa \in (0,\alpha)$, we have $C^{1-\alpha+\kappa} \subset W_T^{1-\alpha, \infty} \subset C^{1-\alpha-\kappa}$.}	

For $h \in W_T^{1-\alpha, \infty}$, we have
$$
\begin{aligned}
\Lambda_\alpha(h) :=\frac{1}{\Gamma(1-\alpha)} \sup _{0<s<t<T}\mid\left( {D_{t-}^{1-\alpha} h_{t-}}\right)(s) \mid \leq \frac{1}{\Gamma(1-\alpha) \Gamma(\alpha)}\|h\|_{1-\alpha, \infty, T}<\infty.
\end{aligned}
$$
Moreover, if $g \in W_0^{\alpha, 1}([0,T], \mathbb{R}^{l\times k})$ 
and $h \in W_T^{1-\alpha, \infty}([0,T], \mathbb{R}^{k})$, the integral $\int_0^t g dh \in \mathbb{R}^{l}$ is well-defined in the sense that 
\begin{eqnarray*}
	\int_a^b g d h=(-1)^\alpha \int_a^b D_{a+}^\alpha g_{a+}(x) D_{b-}^{1-\alpha}h_{b-}(x) d x+g(a+)(h(b-)-h(a+)),
\end{eqnarray*}
where $g(a+):=\lim _{\varepsilon \searrow 0} g(a+\varepsilon)$,  {$h(b-):=\lim _{\varepsilon \searrow 0} h(b-\varepsilon)$}, $g_{a+}(x):=(g(x)-g(a+)) \mathbf{1}_{(a, b)}(x),
h_{b-}(x):=(h(x)-h(b-)) \mathbf{1}_{(a, b)}(x)$.
Recall that the Weyl derivatives of $g$ and $h$ are defined respectively as follows:
$$
D_{a+}^\alpha g(t):=\frac{1}{\Gamma(1-\alpha)}\left(\frac{g(t)}{(t-a)^\alpha}+\alpha \int_a^t \frac{g(t)-g(s)}{(t-s)^{\alpha+1}} d s\right) \mathbf{1}_{(a, b)}(t)
$$
and 
$$
D_{b-}^{1-\alpha} h_{b-}(t):=\frac{(-1)^{1-\alpha}}{\Gamma(\alpha)}\left(\frac{h(t)-h(b)}{(b-t)^{1-\alpha}}+({1-\alpha}) \int_t^b \frac{h(t)-h(s)}{(s-t)^{{2-\alpha}}} d s\right) \mathbf{1}_{(a, b)}(t)
$$
for almost $t\in(a,b)$ (the convergence of the integrals at the
singularity $s = t$ holds pointwise for almost all $t \in(a, b) $ if $p = 1$ and moreover in
$L^p$-sense if $1 < p < \infty$). Here, $\Gamma$ denotes the Gamma funciton.
So the following estimate 
 {\begin{eqnarray}\label{Lamada}
		\left|\int_s^t g d h\right| \leq \Lambda_\alpha(h) \int_s^t\left(\frac{|g(r)|}{(r-s)^\alpha}+ \alpha\int_s^r \frac{|g(r)-g(q)|}{(r-q)^{\alpha+1}}dq\right) d r,
\end{eqnarray}}
holds for all $t \in [0,T]$.


\subsection{Fractional Brownian motion}\label{sec-2-2}
{
	We introduce a  FBM
	with Hurst parameter $H$ and recall some basic knowledge  for later use.
	In this and next subsection, $H \in (0,1)$ unless otherwise stated. 
	In what follows, 
	$d_1, d_2 \in \mathbb{N}$ and  $d:=d_1+ d_2$.
}

Consider an $\mathbb{R}^{d_1}$-valued FBM as follows
\[
(B^H_t)_{t\in [0,T]}=(B_t^{H,1},B_t^{H,2},\cdots,B_t^{H,d_1})_{t\in [0,T]}
\]
with Hurst parameter $H\in (0,1)$,
where $(B_t^{H,i})_{t\in [0,T]}$, $1 \le i\le d_1$,
are independent one-dimensional FBMs.
It is a centered Gaussian process characterized by the following covariance formula
\[
\mathbb{E}\big[B_{t}^{H,i} B_{s}^{H,j}\big]
=\frac{\delta_{ij}}{2}\left[t^{2 H}+s^{2 H}-|t-s|^{2 H}\right], 
\qquad  s, t \in [0,T], \, 1 \le i, j\le d_1.
\]
Here, $\delta_{ij}$ stands for Kronecker's delta.
The increment satisfies that 
\[
\mathbb{E}\big[(B_{t}^{H,i}-B_{s}^{H,i})(B_{t}^{H,j}-B_{s}^{H,j})\big]=\delta_{ij}|t-s|^{2 H}, 
\qquad   s, t \in [0,T], \, 1 \le i, j\le d_1.
\]
When $H=1/2$, it is a standard $\mathbb{R}^{d_1}$-valued  BM.
According to the
Kolmogorov's continuity criterion, 
$B^H$ almost surely has $H'$-H\"older continuous trajectories for each $H'\in(0,H)$. 

 {Then according to the result in \cite[Corollary 3.1]{1999Decreusefond},}  and by using a standard $\mathbb{R}^{d_1}$-valued BM
$B=(B_t)_{t\in [0,T]}$, we can construct an  FBM $B^H=(B_t^H)_{t\in [0,T]}$ in the following way:
\begin{equation}\label{vol-rep}
B_t^H :=\int_0^T K_H(t, s) d B_s, \qquad t \in[0,T],
\end{equation}
where we set for all $0\le s \le t \le T$
$$
K_H(t, s) :=k_H(t, s) 1_{[0, t]}(s) 
$$
with
$$
k_H(t, s) :=\frac{c_H}{\Gamma\left(H+\frac{1}{2}\right)}(t-s)^{H-\frac{1}{2}} F\left(H-\frac{1}{2}, \frac{1}{2}-H, H+\frac{1}{2} ; 1-\frac{t}{s}\right),
$$
where
$c_H=\left[\frac{2 H \Gamma\left(\frac{3}{2}-H\right) \Gamma\left(H+\frac{1}{2}\right)}{\Gamma(2-2 H)}\right]^{1 / 2}$,   $F$ is the Gauss hypergeometric function. 
(For more details, see \cite{2020Budhiraja}.)
It is known that the map $B \mapsto B^H$  defined by (\ref{vol-rep}) is
a measurable isomorphism from the classical Wiener space to
the abstract Wiener space of FBM.
Hence, the two Gaussian structures will be 
identified through this map.

Moreover, when  $H >1/2$ and $1- H < \alpha < 1/2 $,   we have the following result. See \cite[Lemma 7.5]{2002Rascanu}  for a proof.
\begin{prop} \label{p1}
	Let $1/2<H<1$ and $1- H < \alpha < 1/2 $. Then,    $(B^H_t)_{t\in [0,T]}\in W_T^{1-\alpha, \infty} ([0,T], \mathbb{R}^{d_1})$ almost surely. Moreover, if a  stochastic process ${(v_t)_{t \in [0, T ]}}\in W_0^{\alpha, 1}([0,T], \mathbb{R}^{m\times d_1})$, then the generalized Riemann-Stieltjes
	integral $\int_0^t v_s d B_s^H$ exists 
	and
	$$
	\left|\int_0^t v_s d B_s^H\right| \leq \Lambda_\alpha\left(B^H\right)\|v\|_{\alpha, 1},
	\quad t\in[0,T].$$
	Here,   {$\Lambda_\alpha\left(B^H\right):=\frac{1}{\Gamma(1-\alpha)} \sup _{0<s<t<T}\left|\left(D_{t-}^{1-\alpha} B_{t-}^H\right)(s)\right| \leq \frac{1}{\Gamma(1-\alpha) \Gamma(\alpha)}\left\|B^H\right\|_{1-\alpha, \infty, T}$}, which has moments of all order.
\end{prop}

Then we  introduce the Cameron-Martin space for   FBM with $H\in(0,1)$. For $f\in L^2([0,T],\mathbb{R}^{d_1})$, we define
\[
\mathcal{K_H}f(t)=\int_0^T K_H(t, s)f(s) d s, \qquad t \in[0,T].
\]	
The Cameron-Martin Hilbert space 
$\mathcal{H}^{H,d_1}= \mathcal{H}^{H}([0,T], \mathbb{R}^{d_1})$
for $(B^H_t)_{0\le t\le T}$ is  defined by
\[
\mathcal{H}^{H,d_1}=\big\{\mathcal{K_H} \dot{h}: \dot{h}\in L^2\left([0,T], \mathbb{R}^{d_1}\right)\big\}
\]
(Note that $\dot h$ is {\it not} the time derivative of $h$.) 
It holds that $\mathcal{H}^{H,d_1} \subset C^{H'}([0,T], \mathbb{R}^{d_1})$ for all $H' \in (0,H)$.
The  inner product on  $\mathcal{H}^{H,d_1}$ is  defined by 
\[
\langle h, g\rangle_{\mathcal{H}^{H,d_1}}=\langle \mathcal{K_H} \dot{h},\mathcal{K_H} \dot{g}\rangle_{\mathcal{H}^{H,d_1}}:=\langle\dot{h}, \dot{g}\rangle_{L^2}.
\]
Note that $\mathcal{K_H}:L^2([0,T],\mathbb{R}^{d_1}) \to {\mathcal{H}^{H,d_1}}$ is a unitary isometry, more explanation can be found in {\cite{1999Decreusefond}}.

For $N\in \mathbb{N}$, we set 
\[
S_N=\left\{\phi \in \mathcal{H}^{H,d_1}: \frac{1}{2} 
\|\phi\|_{\mathcal{H}^{H,d_1}}^{2} 
\leq N\right\}. 
\]
Under the weak topology, the ball $  S_N$ is actually  a compact Polish space. 		

Next, we introduce the probability space for FBM. Let 	$(\Omega, \mathcal{F}, \mathbb{P})$ be the $d_1$-dimensional 
classical Wiener space, where $\Omega=C_0 ([0,T], \mathbb{R}^{d_1} )$
is the space of $\mathbb{R}^{d_1}$-valued continuous paths 
starting at $0$ under the uniform topology, 
$\mathbb{P}$ is
the  Wiener measure on $\Omega$ and 
$\mathcal{F}$ is the $\mathbb{P}$-completion of
the Borel $\sigma$-field of $\Omega$. The filtration $\{\mathcal{F}_t\}_{t\in[0,T]} $ is defined by $\mathcal{F}_t :=\sigma \{B_s: 0 \leq s \leq t\} \vee \mathcal{N}$, where $\mathcal{N}$ is the set of all $\mathbb{P}$-negligible events. This filtration satisfies the usual condition. It is known that if $B^H$ is defined by  (\ref{vol-rep}), then we have  $\mathcal{F}_t=\sigma \{B_s^H: 0 \leq s \leq t\} \vee \mathcal{N}$.

We denote  $\mathcal{A}_b^N$, $N\in \mathbb{N}$,  the set of all 
$\mathbb{R}^{d_1}$-valued  processes $(\phi_t)_{t \in [0,T]}$ on 
$(\Omega, \mathcal{F}, \mathbb{P})$ of the form  	
$
\phi=\mathcal{K}_H \dot{\phi},
$ 
{$\mathbb{P}$-a.s.}
for some $\{\mathcal{F}_t\}$-progressively measurable process $\dot{\phi}={(\dot{\phi_t})}_{t\in[0,T]}$ such that $\frac12 \int_0^T |\dot{\phi}_t|^2_{\mathbb{R}^{d_1}} dt  \le N$, $\mathbb{P}$-a.s.
We set $\mathcal{A}_b =\cup_{N\in \mathbb{N}} \mathcal{A}_b^N$.
Due to that every $\phi\in  \mathcal{A}_b^N$  is an $S_N$-valued random 
variable and the ball $S_N$ is compact, the family $\{ \mathbb{P}\circ \phi^{-1} :  
\phi \in \mathcal{A}_b^N\}$ of probability measures is automatically tight.	
By Girsanov's formula,   
the law of  $B^H +\phi$ is 
mutually absolutely continuous to  that of  $B^H$ for every $\phi \in \mathcal{A}_b$.

\subsection{Mixed Fractional Brownian motion}\label{sec-2-3}
In this subsection, we introduce a mixed FBM
of Hurst parameter $H\in (0,1)$.

Firstly, consider an  $\mathbb{R}^{d_2}$-valued standard  BM $(W_t)_{t\in [0,T]}$ which is assumed to be independent of FBM $(B^H_t)_{t\in [0,T]}$.
The Cameron-Martin Hilbert space   {$\mathcal{H}^{\frac{1}{2},d_2} =\mathcal{H}^{\frac{1}{2}}([0,T], \mathbb{R}^{d_2})$} 
for  $(W_t)_{t\in  [0,T]}$  is defined by
\[
\mathcal{H}^{\frac{1}{2},d_2}:=\big\{  v _{\cdot}=\int_{0}^{\cdot}  v^{\prime} _{s} d s \in C([0,T],\mathbb{R}^{d_2}): v^{\prime} \in L^2([0,T],\mathbb{R}^{d_2})\big\}.
\]
The inner product of $\mathcal{H}^{\frac{1}{2},d_2}$ is defined by $\langle v, w \rangle_{\mathcal{H}^{\frac{1}{2},d_2}}:=
\int_{0}^{T} \langle v^{\prime} _{t}, w^{\prime} _{t}\rangle_{\mathbb{R}^{d_2}} d t$.

The $\mathbb{R}^d$-valued 
process $(B^H_t ,W_t)_{t\in [0,T]}$ is called  a mixed FBM 
of Hurst parameter $H$, where $d=d_1+d_2$.
Then $\mathcal{H}:={{\mathcal{H}^{H,d_1}}\oplus{\mathcal{H}^{\frac{1}{2},d_2}}}$ is the Cameron-Martin space for   $(B_t^H, W_t)_{0\le t\le T}$. 
For $N\in\mathbb{N}$, we define 
\[
\tilde S_N=\left\{(\phi,\psi) \in \mathcal{H}: 
\frac{1}{2}\|(\phi,\psi)\|_{\mathcal{H}}^2 := \frac{1}{2} 
(\|\phi\|_{\mathcal{H}^{H,d_1}}^{2} +\|\psi \|_{\mathcal{H}^{\frac{1}{2},d_2}}^{2})
\leq N\right\}. 
\]
Under the weak topology, the ball $ \tilde S_N$ is actually  a compact Polish space. 	

Next, we introduce the probability space for mixed FBM. Let 	$(\tilde\Omega, \tilde{\mathcal{F}}, \tilde{\mathbb{P}})$ be the $d$-dimensional 
classical Wiener space, where $\tilde\Omega=C_0 ([0,T], \mathbb{R}^{d} )$
is the space of $\mathbb{R}^{d}$-valued continuous paths 
starting at $0$ under the uniform topology, 
$\tilde{\mathbb{P}}$ is
the  Wiener measure on $\tilde\Omega$ and 
$\tilde{\mathcal{F}}$ is the $\tilde{\mathbb{P}}$-completion of
the Borel $\sigma$-field of $\tilde\Omega$. The filtration $\{\tilde{\mathcal{F}}_t\}_{t\in[0,T]} $ is defined by $\tilde{\mathcal{F}}_t :=\sigma \{(B_s, W_s): 0 \leq s \leq t\} \vee \mathcal{N}$, where $\mathcal{N}$ is the set of all $\tilde{\mathbb{P}}$-negligible events.  {This filtration satisfies the usual conditions.} It is known that if $B^H$ is defined by  (\ref{vol-rep}),  { then we have  $\tilde{\mathcal{F}_t}=\sigma \{(B_s^H, W_s): 0 \leq s \leq t\} \vee \mathcal{N}$.}

We denote  $\tilde{\mathcal{A}}_b^N$, $N\in \mathbb{N}$,  the set of all 
$\mathbb{R}^{d}$-valued  processes $(\phi_t,\psi_t)_{t \in [0,T]}$ on 
$(\tilde\Omega, \tilde{\mathcal{F}}, \tilde{\mathbb{P}})$ of the form  	
$(\phi, \psi)=\Bigl( \mathcal{K}_H \dot{\phi}, \, \int_0^{\cdot} \psi^{\prime}_s ds \Bigr)$, 
{$\tilde{\mathbb{P}}$-a.s.}
for some $\{\tilde{\mathcal{F}}_t\}$-progressively measurable process $(\dot{\phi}, \psi^{\prime})$ such that $\frac12 \int_0^T  ( |\dot{\phi}_t|^2_{\mathbb{R}^{d_1}}+ |\psi^{\prime}_t|^2_{\mathbb{R}^{d_2}} )  dt  \le N$, $\tilde{\mathbb{P}}$-a.s.
We set $\tilde{\mathcal{A}}_b =\cup_{N\in \mathbb{N}} \tilde{\mathcal{A}}_b^N$.
Due to that every $(\phi,\psi)\in  \tilde{\mathcal{A}}_b^N$  is an $\tilde{S}_N$-valued random 
variable and the ball $\tilde{S}_N$ is compact, the family $\{ \mathbb{P}\circ (\phi,\psi)^{-1} :  
(\phi,\psi) \in \tilde{\mathcal{A}}_b^N\}$ of probability measures is automatically tight.	
By Girsanov's formula,   
the law of  $(B^H +\phi,W+\psi)$ is 
mutually absolutely continuous to  that of  $(B^H,W)$ for every $(\phi,\psi) \in \tilde{\mathcal{A}}_b$.

\section{MDP for the single-time scale system}
In this section, we first focus on the single-time scale system with FBM (\ref{0}). In what follows, $1/2<H< 1$ will always be assumed.  	
We assume $\varepsilon\in(0,1]$ and  will let  $\varepsilon \to 0$.  In the following subsection, we will introduce assumptions of our main theorem and then state our first main theorem.
\subsection{Assumptions and  Statement of Main Result}


It is easy to see that as $\varepsilon \to 0$,  {the solution $x^\varepsilon$ to (\ref{0}) converges to the solution of the following deterministic ODEs}:
\begin{eqnarray}\label{0-1}
dx_t = f(x_t)dt , x_0\in \mathbb{R}^{m}.
\end{eqnarray}
Moreover, $x\in C^1([0,T],\mathbb{R}^m)$ since we will assume $f$ is continuous.
The deviation component $z^\varepsilon$ defined on (\ref{02}) satisfies the following SDE:
\begin{eqnarray}\label{3-1}
dz^{\varepsilon}_t = \frac{1}{\sqrt{\varepsilon}h(\varepsilon)}\big(f(x^{\varepsilon}_t)-f(x_t)\big)dt + \frac{1}{h(\varepsilon)}  \sigma( x^{\varepsilon}_t)dB^H_{t},\quad z^{\varepsilon}_0=0.
\end{eqnarray}

Firstly, we need to  ensure the existence and uniqueness of solution to the system (\ref{0}). To this end, we impose the following conditions:
\begin{itemize}
	\item[(\textbf{A1}).]   Assume  $\sigma\in C^1$ and there exists a  constant $L> 0$ such that for  any $ x_1 , x_2\in \mathbb{R}^{m}$, 
	$$\left|D\sigma\left( x_1\right)\right| \leq L,	\qquad
	\left|D\sigma\left( x_1\right)-D\sigma\left( x_2\right)\right| \leq L\left|x_1-x_2\right|$$
	hold. Here, $D$ is the standard gradient operator on $\mathbb{R}^{m}$.
	\item[(\textbf{A2}).] Assume that there exists a constant $L> 0$ such that for  any $ x_1$,  $ x_2\in \mathbb{R}^{m} $, 
	\begin{equation*}
	\begin{aligned}
	&\left|f\left( x_1\right)-f\left( x_2\right)\right| 
	\leq L\left|x_1-x_2\right|, 
	\end{aligned}
	\end{equation*}	
	hold.
\end{itemize}	
Clearly, (\textbf{A1}) and (\textbf{A2}) imply that $\sigma$ and $f$ are of   linear growth, that is, there exists a constant $L^{\prime}>0$ such that for any $x\in \mathbb{R}^m$, 	
\begin{eqnarray}\label{2.6}
\left|\sigma\left( x\right)\right| +\left|f\left( x\right)\right|\leq L^{\prime}\left(1+\left|x\right|\right)
\end{eqnarray}
holds.

By  \cite[Theorem 2.2]{Guerra}, under  (\textbf{A1}) and (\textbf{A2}) above,   there exists  a unique (pathwise) solution $x^{\varepsilon}$ to the system (\ref{0}). Then  $z^\varepsilon$ will be defined by $x^{\varepsilon}$. Consequently, there is a measurable map
\[
\mathcal{G}^{\varepsilon} (\bullet/ h(\varepsilon)): C_0\left([0,T], \mathbb{R}^{d_1}\right) \rightarrow C^{\alpha}\left([0,T], \mathbb{R}^m\right)
\]
{with $\alpha\in (1-H,1/2)$} such that
$z^{\varepsilon}:=\mathcal{G}^{\varepsilon}( B^H/h(\varepsilon) )$. In other words, this map is a solution map of (\ref{3-1}).

In order to study an MDP for the system (\ref{0}),
we  impose more conditions: 	
\begin{itemize}
	\item[(\textbf{A3}).]   Assume that  $f\in C^1$, and there exists a  constant $L> 0$ such that for  any $ x_1 , x_2\in \mathbb{R}^{m}$, 
	$$\left|Df\left( x_1\right)\right| \leq L,	\qquad
	\left|Df\left( x_1\right)-Df\left( x_2\right)\right| \leq L\left|x_1-x_2\right|$$
	hold. 
\end{itemize}			
Then, it is easy to see  that $x\in C^1([0,T],\mathbb{R}^m)$, where $x$ solves (\ref{0-1}).

Now, we define the skeleton equation as follows
\begin{eqnarray}\label{3'}
d\tilde{z}_t = {Df}({x}_t)\tilde{z}_tdt + \sigma( {x}_t)du_t,\quad\tilde{z}_0=0.
\end{eqnarray}	
Then (\textbf{A3}) implies that for any fixed $x\in \mathbb{R}^m$, $z \mapsto Df(x)z$ is linear from $\mathbb{R}^m$ to $\mathbb{R}^m$ (hence    Lipschitz continuous). 
For $u \in S_{N}$,  there exists a unique solution $\tilde {z} \in  W_0^{\alpha, \infty}([0,T], \mathbb{R}^{m})$  to the skeleton equation (\ref{3'}).  Moreover, the following estimate
$$\|\tilde{z}\|_{1-\alpha}\le c,$$ holds, where the constant $c=c_N$ is independent of $u\in S_N$. For a   precise proof, see  \cite[Proposition 3.6]{2020Budhiraja}. Due to the condition that  $1-\alpha>\alpha$, we have the  compact embedding $C^{1-\alpha}([0,T],\mathbb{R}^{m}) \subset C^{\alpha}([0,T],\mathbb{R}^{m})$. 
We  also define a map
\[
\mathcal{G}^{0}: \mathcal{H}^{H,d_1} \rightarrow  C^{\alpha}\left([0,T], \mathbb{R}^m\right)
\]
by  $\tilde{z}=\mathcal{G}^{0}(u)$.


Now, we provide a precise statement of our first main theorem. 	By definition,  an MDP for $\{x^\varepsilon\}_{\varepsilon\in(0,1]}$ is  equivalent to an LDP for that $\{z^\varepsilon\}_{\varepsilon\in(0,1]}$, so we will prove the latter.
\begin{thm}\label{thm1}
	Let $H\in(1/2,1)$  and $\alpha\in (1-H,1/2)$ and assume Assumptions (\textbf{A1})--(\textbf{A3}).
	Then,  as $\varepsilon \to 0$, the  solution $\{x^\varepsilon\}_{\varepsilon\in(0,1]}$ to   { the equation (\ref{0}) } satisfies an MDP  with speed $b(\varepsilon)=1/h^2(\varepsilon)$ on the  H\"older path space 
	$C^{\alpha}\left([0,T], \mathbb{R}^m\right)$ with the good rate function $I: C^{\alpha}\left([0,T], \mathbb{R}^m\right)\rightarrow [0, \infty]$   defined by
	\begin{eqnarray*}\label{rate1}
		I(\xi) &=& {
			\inf\Big\{\frac{1}{2}\|u\|^2_{\mathcal{H}^{H,d_1}}~:~{ u\in \mathcal{H}^{H,d_1} 
				\quad\text{such that} \quad\xi =\mathcal{G}^{0}(u)}\Big\} 
		}.		\end{eqnarray*}
\end{thm}	
\begin{rem} {
The H\"older space is not a separable space hence not a Polish space. Therefore, we cannot use variational formula on the classical Wiener space $(\Omega, \mathcal{F}, \mathbb{P})$ with $\Omega=C^\alpha ([0,T], \mathbb{R}^{k} )$ directly. But, there is a famous  Banach subspace $H^\alpha = H^\alpha  ([0,T], \mathbb{R}^{k})$, which  is the space that for all $g\in C^\alpha([0,T], \mathbb{R}^{k})$,  with the norm
\begin{eqnarray*}\label{111}
	\lim_{\delta\to0+}\sup_{\substack{|t-s|\le \delta\\0 \leq s<t \le T}} \frac{|g(t)-g(s)|}{(t-s)^\alpha}=0. 
\end{eqnarray*}
The space $H^\alpha$ is sometimes called the  ``little H\"older space" informally. According to \cite{Ciesielski}, we obtain that the space $H^\alpha$ is a separable space and 
$$
	H^\alpha=\overline{\bigcup_{\kappa>0}C^{\alpha+\kappa}},
$$
where the closure is taken in the sense of the norm $\|\cdot\|_{\alpha\textrm{-hld}}$.}

 {For any given $\alpha$ that satisfies the assumption on the Holder exponent,
we can find slightly large $\beta$ which still satisfies the condition, for instance we take $\beta=\alpha+\varepsilon$.
Then, our process $x^\varepsilon$ takes values in $\alpha+\varepsilon$-H\"older space,
which is a subset of the little $\alpha$-H\"older space, so we could show that $x^\varepsilon\in H^\alpha$, which is a separable  Banach space.
We can apply   variational method 
on the $H^\alpha$ space,
which is continously embedded in the usual $\alpha$-H\"older space. Then, we will prove the weak convergence method with the $\alpha$-H\"older norm on the closed subset.
By the contraction principle, the same LDP holds on the usual  $\alpha$-H\"older space, too. }
\end{rem}
\subsection{Main Proof}	
We will prove the LDP for $\{z^\varepsilon\}_{\varepsilon\in(0,1]}$ with speed $b(\varepsilon)=1/h^2(\varepsilon)$ on the  H\"older path space 
$C^{\alpha}\left([0,T], \mathbb{R}^m\right)$ with the good rate function $I$.
%

Before proving \thmref{thm1}, we give some prior estimates. Firstly, for a  control $u^{\varepsilon}\in \mathcal{A}^{b}$, by making a shift $B^H \mapsto B^H+h(\varepsilon)u^\varepsilon$, we consider the following controlled system associated to (\ref{0}):
\begin{eqnarray}\label{3.2}
d\tilde {x}^{\varepsilon}_t =f(\tilde {x}^{\varepsilon}_t)dt + \sqrt \varepsilon h(\varepsilon)\sigma(\tilde {x}^{\varepsilon}_t)du^{\varepsilon}_t+ \sqrt \varepsilon  \sigma(\tilde {x}^{\varepsilon}_t)dB^H_{t},\quad \tilde x^{\varepsilon}_0=x_0\in \mathbb{R}^{m}.
\end{eqnarray}
Respectively, $\tilde z^\varepsilon$ satisfies the following controlled SDEs:
\begin{eqnarray}\label{3.3}
d\tilde {z}^{\varepsilon}_t =\frac{1}{\sqrt{\varepsilon}h(\varepsilon)}\big(f(\tilde x^{\varepsilon }_t)-f(x_t)\big)dt + \sigma(\tilde {x}^{\varepsilon}_t)du^{\varepsilon}_t+ \frac{1 }{h(\varepsilon)} \sigma(\tilde {x}^{\varepsilon}_t)dB^H_{t},\quad \tilde z^{\varepsilon}_0=0.
\end{eqnarray}
It is not hard to verify that there exists a unique solution $\tilde {z}^{\varepsilon}=\mathcal{G}^{\varepsilon} (\frac{B^H}{h(\varepsilon)}+u^{\varepsilon})$ of the system (\ref{3.3}). For the rest of the section, we will often use the estimate (\ref{Lamada}).
\begin{lem}\label{lem3-1}
	Assume (\textbf{A1})--(\textbf{A2}) and let  $p\ge1$ and  $N\in\mathbb{N}$. Then, there exists   a constant $C>0$ such that for every $u^{\varepsilon}\in \mathcal{A}_b^N$, we have
	\begin{equation*}
	\mathbb{E}\big[\|\tilde x^{\varepsilon}\|_{\alpha, \infty}^p\big] \leq C,
	\end{equation*}
	where $C>0$ depends only on $p$ and $N$.
\end{lem}	
\begin{proof}
	In this proof, $C$ is a positive constant depending on only $p$ and $N$, which may change from line to line.   Firstly, denote $\Lambda:=\Lambda_\alpha\left(B^H\right) \vee 1$, and for any $\lambda\ge 1$, set 
	$$
	\|g\|_{\lambda, t}:=\sup _{0 \leq s \leq t} e^{-\lambda s}|g(s)|
	$$
	and 
	$$\|g\|_{1, \lambda, t}:=\sup _{0 \leq s \leq t} e^{-\lambda s} \int_0^s \frac{|g(s)-g(r)|}{(s-r)^{\alpha+1}} d r.$$ 
	
 {By  (\textbf{A1}), (\textbf{A2}) and \eqref{2.6}, the condition that $\sqrt{\varepsilon}h(\varepsilon)\in(0,1]$ and  \eqref{Lamada}}, we have
	\begin{equation}\label{3-2}
	\begin{aligned}
	\|\tilde x^{\varepsilon}\|_{\lambda, t} &=\sup _{0 \leq s \leq t} e^{-\lambda s}\left|x_0+\int_0^s f(\tilde x_r^{\varepsilon}) d r+\sqrt{\varepsilon}h(\varepsilon)\int_0^s \sigma( \tilde x_r^{\varepsilon}) d u_r^\varepsilon +\sqrt{\varepsilon}\int_0^s \sigma( \tilde x_r^{\varepsilon}) d B_r^H\right| \cr
	& \leq C \Lambda\big[1+\sup _{0 \leq s \leq t} \int_0^s e^{-\lambda(s-r)}\big(r^{-\alpha}\|\tilde x^{\varepsilon}\|_{\lambda, t}+\|\tilde x^{\varepsilon}\|_{1, \lambda, t}\big) d r\big] \cr
	& \quad+ C\|u^\varepsilon\|_{\mathcal{H}^{H,d_1}} \big[1+\sup _{0 \leq s \leq t} \int_0^s e^{-\lambda(s-r)}\big(r^{-\alpha}\|\tilde x^{\varepsilon}\|_{\lambda, t}+\|\tilde x^{\varepsilon}\|_{1, \lambda, t}\big) d r\big] ,
	\end{aligned}
	\end{equation}
	where the  inequality is due to the fact that  $\Lambda_\alpha(u^\varepsilon) \le  C\|u^\varepsilon\|_{H}\le C \|u^\varepsilon\|_{\mathcal{H}^{H,d_1}}$. See \cite[Lemma 4.1]{2020Budhiraja}.
	
	Since we have the following simple estimates
	\begin{eqnarray}\label{3-3}
	\int_0^t e^{-\lambda(t-r)} r^{-\alpha} d r 
	\leq \lambda^{\alpha-1} \sup _{z>0} \int_0^z e^{-y}(z-y)^{-\alpha} d y 
	\leq C \lambda^{\alpha-1},
	\end{eqnarray}
	we can easily see  that 
	\begin{eqnarray}\label{3-4}
	\|\tilde x^{\varepsilon}\|_{\lambda, t} 
	&\le&  {C(\Lambda+\|u^\varepsilon\|_{\mathcal{H}^{H,d_1}})\big(1+\lambda^{\alpha-1}\|\tilde x^{\varepsilon}\|_{\lambda, t}+\lambda^{-1}\|\tilde x^{\varepsilon}\|_{1, \lambda, t}\big)}.
	\end{eqnarray}
	
	Next, we will  compute $\|\tilde x^{\varepsilon}\|_{1, \lambda, t}$. Before doing so, we give some  prior estimates. Let
	\begin{eqnarray}\label{3-5}
	\mathcal{C}_1:=\int_0^t(t-s)^{-\alpha-1}\left|\int_s^t g(r) d B_r^H\right| d s.
	\end{eqnarray}
	By  Fubini's theorem and  straightforward computation, it holds that
	\begin{eqnarray*}\label{3-6}
		\mathcal{C}_1 &\le & \Lambda\left(\int_0^t \int_s^t(t-s)^{-\alpha-1} \frac{|g(r)|}{(r-s)^\alpha} d r d s+\int_0^t \int_s^t \int_s^r(t-s)^{-\alpha-1} \frac{|g(r)-g(q)|}{(r-q)^{1+\alpha}} d q d r d s\right) \cr
		&\le& \Lambda\left(\int_0^t \int_0^r(t-s)^{-\alpha-1}(r-s)^{-\alpha} d s|g(r)| d r+\int_0^t \int_0^r \int_0^q(t-s)^{-\alpha-1} d s \frac{|g(r)-g(q)|}{(r-q)^{1+\alpha}} d q d r\right) \cr
		&\le&\Lambda\left(c_\alpha \int_0^t(t-r)^{-2 \alpha}|g(r)| d r+\int_0^t \int_0^r(t-q)^{-\alpha} \frac{|g(r)-g(q)|}{(r-q)^{1+\alpha}} d q d r\right),
	\end{eqnarray*}
	where $c_\alpha=\int_0^{\infty}(1+q)^{-\alpha-1} q^{-\alpha} d q=B(2 \alpha, 1-\alpha)$ and $B$ is the well-known Beta function.
	
	Similarly, we  define  
	\begin{eqnarray*}\label{3-7}
		\mathcal{C}_2:=\int_0^t(t-s)^{-\alpha-1}\left|\int_s^t g(r) d u_r^\varepsilon\right| d s.
	\end{eqnarray*}
	Due to the property of the control $u^\varepsilon$ that $\Lambda_\alpha(u^\varepsilon) \le  C\|u^\varepsilon\|_{H}\le C \|u^\varepsilon\|_{\mathcal{H}^{H,d_1}}$,  
	\begin{eqnarray}\label{3-8}
	\mathcal{C}_2 
	&\le& {C}\|u^\varepsilon\|_{\mathcal{H}^{H,d_1}}\left(c_\alpha \int_0^t(t-r)^{-2 \alpha}|g(r)| d r\right.\cr
	&&\left.+\int_0^t \int_0^r(t-q)^{-\alpha} \frac{|g(r)-g(q)|}{(r-q)^{1+\alpha}} d q d r\right).
	\end{eqnarray}
	 {Then, we define $\mathcal{C}_3:=\sup _{0 \leq s \leq t} e^{-\lambda s} \int_0^s(s-r)^{-\alpha-1}\left|\int_r^s f\left(\tilde{x}_q^{\varepsilon}\right) d q\right| d r$.
		By  (\textbf{A2}), we obtain
\begin{eqnarray*}
	 \mathcal{C}_3 
	&\leq & C \sup _{0 \leq s \leq t} e^{-\lambda s} \int_0^s(s-r)^{-\alpha-1} \int_r^s\left(1+\left|\tilde{x}_q^{\varepsilon}\right|\right) d q d r \\
	&\leq & C\left(1+\sup _{0 \leq s \leq t} e^{-\lambda s} \int_0^s(s-r)^{-\alpha-1} \int_r^s\left|\tilde{x}_q^{\varepsilon}\right| d q d r\right) \\
	&= & C\left(1+\sup _{0 \leq s \leq t} e^{-\lambda s} \int_0^s(s-r)^{-\alpha-1} \int_r^s(q-r)^\alpha \frac{\left|\tilde{x}_q^{\varepsilon}\right|}{(q-r)^\alpha} d q d r\right) \\
&	\leq & C\left(1+\sup _{0 \leq s \leq t} e^{-\lambda s} \int_0^s(s-r)^{-\alpha-1} \int_r^s \frac{\left|\tilde{x}_q^{\varepsilon}\right|}{(q-r)^\alpha} d q d r\right) \\
	&= & C\left(1+\sup _{0 \leq s \leq t} e^{-\lambda s} \int_0^s \int_0^q(s-r)^{-\alpha-1}(q-r)^{-\alpha} d r\left|\tilde{x}_q^{\varepsilon}\right| d q\right) \\
	&\leq & C\left(1+ \sup _{0 \leq s \leq t} e^{-\lambda s} \int_0^s(s-q)^{-2 \alpha}\left|\tilde{x}_q^{\varepsilon}\right| d q\right) \\
	&\leq & C\left(1+ \sup _{0 \leq s \leq t} \int_0^s e^{-\lambda(s-q)}(s-q)^{-2 \alpha}\left\|\tilde{x}_q^{\varepsilon}\right\|_{\lambda, t} d q\right) .
\end{eqnarray*}}
	By some direct computations,  we have
	\begin{eqnarray}\label{3-9}
	\|\tilde x^{\varepsilon}\|_{1, \lambda, t}&=& \sup _{0 \leq s \leq t} e^{-\lambda s} \int_0^s(s-r)^{-\alpha-1}\left|\int_r^s f( \tilde x_q^{\varepsilon}) d q\right| d r \cr
	&&+\sup _{0 \leq s \leq t} e^{-\lambda s} \int_0^s(s-r)^{-\alpha-1}\left|\int_r^s \sqrt{\varepsilon}\sigma( \tilde x_q^{\varepsilon}) d B_q^H\right| d r \cr
	&&+\sup _{0 \leq s \leq t} e^{-\lambda s} \int_0^s(s-r)^{-\alpha-1}\left|\int_r^s \sqrt{\varepsilon}h(\varepsilon)\sigma( \tilde x_q^{\varepsilon}) d u_q^\varepsilon\right| d r \cr
	&\le & C (\Lambda+\|u^\varepsilon\|_{\mathcal{H}^{H,d_1}})\left(1+\sup _{0 \leq s \leq t} \int_0^s e^{-\lambda(s-r)}\times\left[(s-r)^{-2 \alpha}\|\tilde x^{\varepsilon}\|_{\lambda, t}\right.\right.\cr
	&&\left.\left.\qquad\qquad\qquad\qquad\qquad\qquad\qquad\qquad\qquad\qquad+(s-r)^{-\alpha}\|\tilde x^{\varepsilon}\|_{1, \lambda, t}\right] d r\right) \cr
	&\le & C (\Lambda+\|u^\varepsilon\|_{\mathcal{H}^{H,d_1}})\left(1+\lambda^{2 \alpha-1}\|\tilde x^{\varepsilon}\|_{\lambda, t}+\lambda^{\alpha-1}\|\tilde x^{\varepsilon}\|_{1, \lambda, t}\right),
	\end{eqnarray}
	where the first inequality is from Assumptions (\textbf{A1}) and (\textbf{A2}), and the final inequality comes from the following  estimate,
	\begin{eqnarray}\label{3-10}
	\int_0^t e^{-\lambda(t-r)}(t-r)^{-2 \alpha} d r 
	\le \lambda^{2 \alpha-1} \int_0^{\infty} e^{-q} q^{-2 \alpha} d q 
	\le C \lambda^{2 \alpha-1} .
	\end{eqnarray}
	We take $\lambda$  large enough. For instance, we set  $\lambda=\max{\{1,(4 C (\Lambda+\|u^\varepsilon\|_{\mathcal{H}^{H,d_1}}))^{\frac{1}{1-\alpha}}\}}$. From the estimate (\ref{3-4}), we have
	\begin{eqnarray}\label{3-11}
	\|\tilde x^{\varepsilon}\|_{\lambda, t} \leq \frac{4}{3} C (\Lambda+\|u^\varepsilon\|_{\mathcal{H}^{H,d_1}})(1+\lambda^{-1}\|\tilde x^{\varepsilon}\|_{1, \lambda, t}).
	\end{eqnarray}
	Substituting the above estimate into (\ref{3-9}), we can get that
	\begin{eqnarray}\label{3-12}
	\|\tilde x^{\varepsilon}\|_{1, \lambda, t} &\le& \frac{3}{2} C (\Lambda+\|u^\varepsilon\|_{\mathcal{H}^{H,d_1}})+2(C (\Lambda+\|u^\varepsilon\|_{\mathcal{H}^{H,d_1}}))^{1 /(1-\alpha)} \cr
	&\le& C (\Lambda+\|u^\varepsilon\|_{\mathcal{H}^{H,d_1}})^{1 /(1-\alpha)}\cr
	&\le& C (\Lambda+\sqrt{2N})^{1 /(1-\alpha)}.
	\end{eqnarray}
	Substituting (\ref{3-12}) into (\ref{3-11}), we can see that 
	\begin{eqnarray}\label{3-13}
	\|\tilde x^{\varepsilon}\|_{ \lambda, t} 
	&\le& C (\Lambda+\|u^\varepsilon\|_{\mathcal{H}^{H,d_1}})^{1 /(1-\alpha)}
	\le C (\Lambda+\sqrt{2N})^{1 /(1-\alpha)}.
	\end{eqnarray}
	Thus, we have 
	\begin{eqnarray*}\label{3-14}
		\|\tilde x^{\varepsilon}\|_{\alpha, \infty} & \le& e^{\lambda T}\big(\|\tilde x^{\varepsilon}\|_{\lambda, T}+\|\tilde x^{\varepsilon}\|_{1, \lambda, T}\big) \cr
		& \le& C (\Lambda+\|u^\varepsilon\|_{\mathcal{H}^{H,d_1}})^{1 /(1-\alpha)}\exp \big(C (\Lambda+\|u^\varepsilon\|_{\mathcal{H}^{H,d_1}})^{1 /(1-\alpha)}\big)  \cr
		& \le& C(\Lambda+\sqrt{2N})^{1 /(1-\alpha)} \exp \big(C (\Lambda+\sqrt{2N})^{1 /(1-\alpha)}\big) .
	\end{eqnarray*}
	Then, according to the condition that $0<\frac{1}{1-\alpha}<2$, assumptions $u^{\varepsilon}\in \mathcal{A}_b^N$ and the   {Fernique theorem \cite{Fernique}}  for FBM, the conclusion is obtained.
	This proof is completed.
\end{proof}	
\begin{lem}\label{lem3-2}
	Assume (\textbf{A1})--(\textbf{A2}) and let $N\in\mathbb{N}$ and $p\geq1$. Then, there exists a constant $C>0$ such that for every $u^{\varepsilon}\in \mathcal{A}_b^N$, we have
	\begin{equation*}
	\mathbb{E}\big[\sup_{ t\in[0,T]}|\tilde z^{\varepsilon}_t|^p\big] \leq C,
	\end{equation*}
	where $C>0$ depends only on $p$ and $N$.
\end{lem}	
\begin{proof}
	In this proof, $C$ is a positive constant depending on only $p$ and $N$, which may change from line to line.
	By some straightforward computation, we have that
	\begin{eqnarray}\label{3-15}
	|\tilde z^{\varepsilon}_t|^p& \le&   C\big(\left|\int_0^t\frac{f( \tilde x_s^{\varepsilon})-f(  x_s)}{\sqrt{\varepsilon}h(\varepsilon)} d s\right|^p  +\left|\int_0^t\sigma( \tilde x_s^{\varepsilon}) du_s^\varepsilon\right|^p+\left|\int_0^t\frac{\sigma( \tilde x_s^{\varepsilon})}{h(\varepsilon)} dB_s^H\right|^p\big) \cr
	& =:& C(A_1+A_2+A_3),
	\end{eqnarray}
	for some  $C>0$.
	Then, according to Assumption (\textbf{A2}) and H\"older's inequality, we  obtain 
	\begin{eqnarray}\label{3-16}
	A_1& \le& C  \int_0^t|\tilde z^{\varepsilon}_s|^p d s,
	\end{eqnarray}
	for some  $C>0$.

	For the second term $A_2$, by the fact that  $\Lambda_\alpha(u^\varepsilon) \le  C \|u^\varepsilon\|_{\mathcal{H}^{H,d_1}}$, \eqref{3-12} and \eqref{3-13},   we have
	\begin{eqnarray}\label{3-18}
	A_2& \le&
	 {C}\|u^\varepsilon\|_{\mathcal{H}^{H,d_1}}^p\big[ \int_0^ts^{- \alpha}|\sigma(\tilde{x}_s^\varepsilon)| ds+\int_0^t \int_0^s \frac{|\sigma(\tilde{x}_s^\varepsilon)-\sigma(\tilde{x}_r^\varepsilon)|}{(s-r)^{1+\alpha}} d r d s\big]^p\cr
	& \le&  {C}\|u^\varepsilon\|_{\mathcal{H}^{H,d_1}}^p\big[ 1+\int_0^ts^{- \alpha}e^{\lambda T}\|\tilde{x}^\varepsilon\|_{\lambda,t} d s+\int_{0}^{t}e^{\lambda T}  \|\tilde{x}^\varepsilon\|_{1,\lambda,t}ds \big]^p\cr
	& \le& {C}{(2N)}^{\frac{p}{2}}\big[ 1+ (\Lambda+\sqrt{2N})^{1 /(1-\alpha)} \exp \big(C (\Lambda+\sqrt{2N})^{1 /(1-\alpha)}\big) \big]^p,
	\end{eqnarray}
	where $\lambda$ is the same  constant as in  \eqref{3-12} and \eqref{3-13}. 
	
	For the third term $A_3$, by the Assumption (\textbf{A1}), \eqref{3-12}, \eqref{3-13}, and condition that $\frac{1}{h(\varepsilon)}\in(0,1]$, we have
	\begin{eqnarray}\label{3-17}
	A_3& \le&\Lambda^p\big[ \int_0^ts^{- \alpha}|\sigma(\tilde{x}_s^\varepsilon)| d s+\int_0^t \int_0^s \frac{|\sigma(\tilde{x}_s^\varepsilon)-\sigma(\tilde{x}_r^\varepsilon)|}{(s-r)^{1+\alpha}} d r d s\big]^p\cr
	& \le&  {C}\Lambda^p\big[1+ \int_0^ts^{- \alpha}e^{\lambda T}\|\tilde{x}^\varepsilon\|_{\lambda,t} d s+e^{\lambda T}  \|\tilde{x}^\varepsilon\|_{1,\lambda,t} \big]^p\cr
	& \le& {C}\Lambda^p\big[ 1+ (\Lambda+\sqrt{2N})^{1 /(1-\alpha)} \exp \big(C (\Lambda+\sqrt{2N})^{1 /(1-\alpha)}\big) \big]^p.
	\end{eqnarray}

	 {Combine \eqref{3-15}--\eqref{3-17}, and using Gronwall's inequality, we have}
	\begin{eqnarray}\label{3-19}
	|\tilde z^{\varepsilon}_t|^p		& \le&C({(2N)}^{\frac{p}{2}}+\Lambda^p)\big[ 1+ (\Lambda+\sqrt{2N})^{1 /(1-\alpha)} \exp \big(C (\Lambda+\sqrt{2N})^{1 /(1-\alpha)}\big) \big]^p\cr
	&&\cdot e^{CT}.
	\end{eqnarray}
	Then, by taking $\sup_{ t\in[0,T]}$ and taking expectation on the both sides of \eqref{3-19}, using  {Fernique theorem \cite{Fernique}}, we obtain 
	\begin{eqnarray*}\label{3-20}
		\mathbb{E}[\sup_{ t\in[0,T]}|\tilde z^{\varepsilon}_t|^p]& \le&C.
	\end{eqnarray*}
	The proof is completed. 
\end{proof}	

\begin{lem}\label{lem3-3}
	Assume (\textbf{A1})--(\textbf{A2}) and let  $p\ge1$ and  $N\in\mathbb{N}$. Then, there exists a constant $C>0$ such that for every $u^{\varepsilon}\in \mathcal{A}_b^N$, we have
	\begin{equation*}
	\mathbb{E}\big[\|\tilde z^{\varepsilon}\|_{\alpha, \infty}^p\big] \leq C,
	\end{equation*}
	where $C>0$ depends only on $p$ and $N$.
\end{lem}	
\begin{proof}
	In this proof, $C$ is a positive constant depending on only $p$ and $N$, which may change from line to line.	Firstly, still denote $\Lambda:=\Lambda_\alpha\left(B^H\right) \vee 1$.
	 {Similar to \eqref{3-2},	by Assumption (\textbf{A1}), (\textbf{A2}), the condition that $1/h(\varepsilon)\in(0,1]$ and  the estimate  \eqref{Lamada}}, we deduce that
	{\begin{eqnarray}\label{3-2'}
	\|\tilde z^{\varepsilon}\|_{\lambda, t} &=&\sup _{0 \leq s \leq t} e^{-\lambda s}\left|\int_0^s \frac{1}{\sqrt{\varepsilon}h(\varepsilon)}[f(\tilde x_r^{\varepsilon})- f( x_r)]d r+\int_0^s \sigma( \tilde x_r^{\varepsilon}) d u_r^\varepsilon +\frac{1}{h(\varepsilon)}\int_0^s \sigma( \tilde x_r^{\varepsilon}) d B_r^H\right| \cr
	& \leq&   {\sup _{0 \leq s \leq t} e^{-\lambda s}\left|\int_0^s \frac{1}{\sqrt{\varepsilon}h(\varepsilon)}[f(\tilde x_r^{\varepsilon})- f( x_r)]d r\right| }\cr
	&&  {+ \sup _{0 \leq s \leq t} e^{-\lambda s}\left|\int_0^s \sigma( \tilde x_r^{\varepsilon}) d u_r^\varepsilon\right| +  \sup _{0 \leq s \leq t} e^{-\lambda s}\left|\frac{1}{h(\varepsilon)}\int_0^s \sigma( \tilde x_r^{\varepsilon}) d B_r^H\right|} \cr
	& \leq& C \sup _{0 \leq s \leq t}  {|\tilde z^{\varepsilon}_s| } + C \|u^\varepsilon\|_{\mathcal{H}^{H,d_1}}\big[1+\sup _{0 \leq s \leq t} \int_0^s e^{-\lambda(s-r)}\big(r^{-\alpha}\|\tilde x^{\varepsilon}\|_{\lambda, t}+\|\tilde x^{\varepsilon}\|_{1, \lambda, t}\big) d r\big] \cr
	&& + C\Lambda \big[\sup _{0 \leq s \leq t} \int_0^s e^{-\lambda(s-r)}\big(r^{-\alpha}\|\tilde x^{\varepsilon}\|_{\lambda, t}+\|\tilde x^{\varepsilon}\|_{1, \lambda, t}\big) d r\big] \cr
	& \leq& C \sup _{0 \leq s \leq t} {|\tilde z^{\varepsilon}_s| } + C(\sqrt{2N}+\Lambda) \big[1+\|\tilde x^{\varepsilon}\|_{\lambda, t}+\|\tilde x^{\varepsilon}\|_{1, \lambda, t} \big] ,
	\end{eqnarray}}
	where the  second inequality is due to the fact that $\Lambda_\alpha(u^\varepsilon)\le C \|u^\varepsilon\|_{\mathcal{H}^{H,d_1}}$. 	
	
	Next, we will  compute $\|\tilde z^{\varepsilon}\|_{1, \lambda, t}$. 
	By taking  similar computation from (\ref{3-5}) to (\ref{3-8}),  we have
	{\begin{eqnarray}\label{3-9'}
	\|\tilde z^{\varepsilon}\|_{1, \lambda, t}&\le& \sup _{0 \leq s \leq t} e^{-\lambda s} \int_0^s(s-r)^{-\alpha-1}\left|\int_r^s \frac{1}{\sqrt{\varepsilon}h(\varepsilon)}[f( \tilde x_q^{\varepsilon}) -f(  x_q)]d q\right| d r \cr
	&&+\sup _{0 \leq s \leq t} e^{-\lambda s} \int_0^s(s-r)^{-\alpha-1}\left|\int_r^s \sigma( \tilde x_q^{\varepsilon}) d u_q^\varepsilon\right| d r \cr
	&&+\sup _{0 \leq s \leq t} e^{-\lambda s} \int_0^s(s-r)^{-\alpha-1}\left|\int_r^s \frac{1}{h(\varepsilon)}\sigma( \tilde x_q^{\varepsilon}) d B_q^H\right| d r \cr
	&\le & C  \big[\sup _{0 \leq s \leq t}  {|\tilde z^{\varepsilon}_s| } \sup _{0 \leq s \leq t}\int_0^s {(s-r)}^{-\alpha}d r\big] + (\Lambda+\|u^\varepsilon\|_{\mathcal{H}^{H,d_1}})\cr
	& &\left(1+\sup _{0 \leq s \leq t} \int_0^s e^{-\lambda(s-r)}\times\left[(s-r)^{-2 \alpha}\|\tilde x^{\varepsilon}\|_{\lambda, t}+(s-r)^{-\alpha}\|\tilde x^{\varepsilon}\|_{1, \lambda, t}\right] d r\right) \cr
	&\le & C\sup _{0 \leq s \leq t} {|\tilde z^{\varepsilon}_s| } +C (\Lambda+\|u^\varepsilon\|_{\mathcal{H}^{H,d_1}})\left(1+\lambda^{2 \alpha-1}\|\tilde x^{\varepsilon}\|_{\lambda, t}+\lambda^{\alpha-1}\|\tilde x^{\varepsilon}\|_{1, \lambda, t}\right)\cr
	&\le & C\sup _{0 \leq s \leq t}  {|\tilde z^{\varepsilon}_s| } +C (\Lambda+\sqrt{2N})\left(1+\|\tilde x^{\varepsilon}\|_{\lambda, t}+\|\tilde x^{\varepsilon}\|_{1, \lambda, t}\right),
	\end{eqnarray}}
	where the second inequality is from Assumptions (\textbf{A1}),   {(\textbf{A2})} and the condition that $1/h(\varepsilon)< 1$.  The final inequality comes from    the condition $2\alpha<1$.
	Then, by using \eqref{3-12} and \eqref{3-13}, we can see 
	\begin{eqnarray*}\label{3-14'}
		\|\tilde z^{\varepsilon}\|_{\alpha, \infty} & \le&  {C e^{\lambda T} \big(\sup _{0 \leq s \leq t} |\tilde z^\varepsilon_s| +(\Lambda+\sqrt{2N})\big(1+\|\tilde x^{\varepsilon}\|_{ \lambda, t}+\|\tilde x^{\varepsilon}\|_{1, \lambda, t}\big)\big)}\cr
		& \le& C \exp \big(C (\Lambda+\sqrt{2N})^{1 /(1-\alpha)}\big)\cr
		&& { \big(\sup _{0 \leq s \leq t} |\tilde z^\varepsilon_s| +(\Lambda+\sqrt{2N})\big(1+C(\Lambda+\sqrt{2N})^{1/(1-\alpha)}\big)\big)}
	\end{eqnarray*}
	 {By using  \lemref{lem3-1}, \lemref{lem3-2}},  H\"older's inequality and the   {Fernique theorem \cite{Fernique}}, we complete the proof.
\end{proof}
\textbf{Proof of \thmref{thm1}}:

\textbf{Step 1}. 
Let $u^{(j)}$ and $u$ belong to $S_N$ such that $u^{(j)}\rightarrow u$ as $j\rightarrow\infty$ in the weak topology of $\mathcal{H}^{H,d_1}$. 
 {In this step, we proceed} to show the following convergence:
\begin{eqnarray} \label{3-21}
\mathcal{G}^{0}(u^{(j)} )\rightarrow\mathcal{G}^{0}(u) 
\end{eqnarray}
in $C^{\alpha}([0,T],\mathbb{R}^m)$ as $j \to \infty$.

Because $\{u^{(j)}\}_{j\ge 1}\in\mathcal{H}^{H,d_1}$ is bounded, 	by the  {continuous embedding} that 
$\mathcal{H}^{H,d_1}\subset C^{1-\alpha}([0,T], \mathbb{R}^{d_1})$,  
we have that $\{u^{(j)}\}_{j\ge 1} \in C^{1-\alpha}([0,T], \mathbb{R}^{d_1})$ is bounded too. Let $\{\tilde {z}^{(j)}\}_{j\ge 1}$ be a family of solutions to the equation (\ref{3'}), that is,
\begin{eqnarray} \label{3-22}
d\tilde {z}^{(j)}_t =  D{f}(  {x}_t)\tilde z^{(j)}_tdt +  {\sigma( {x}_t)}du^{(j)}_t,\quad \tilde {z}^{(j)}_0=0.
\end{eqnarray}	
It is not too difficult to see that for any fixed $x\in \mathbb{R}^m$, $z \mapsto Df(x)z$ is linear from $\mathbb{R}^m$ to $\mathbb{R}^m$. According to \cite[Proposition 3.6]{2020Budhiraja}, there exists a unique solution to (\ref{3-22}) and $\{\tilde {z}^{(j)}\}_{j\ge 1}\in C^{1-\alpha}([0,T],\mathbb{R}^m)$, moreover,
$$\|\tilde {z}^{(j)}\|_{1-\alpha}\le C,$$
where $C>0$ is independent of $j$.
Due to the compact inclusion  $C^{1-\alpha}([0,T],\mathbb{R}^m)\subset C^{\alpha}([0,T],\mathbb{R}^m)$ for $\alpha<1-\alpha$,  the sequence $\{\tilde {z}^{(j)}\}_{j\ge 1}$ is pre-compact in $C^{\alpha}([0,T],\mathbb{R}^m)$. 
Let $\tilde z$ be any limit point.  {Then, there exists  a subsequence of $\{\tilde z^{(j)}\}_{ j\ge 1}$  converges to $\tilde z$ in $ C^{\alpha}([0,T],\mathbb{R}^m)$ (denoted by the same symbol).}

We will show that the limit point $\tilde z$ satisfies the following differential equation,
\begin{eqnarray} \label{3-44}
d\tilde {z}_t =  D{f}(  {x}_t)\tilde z_tdt + \sigma( {x}_t)du_t \quad \tilde {z}_0=0.
\end{eqnarray}
For any $u=\mathcal{K_H} {\dot u} \in \mathcal{H}^{H,d_1}$, recall that $u$ is differentiable and 
\begin{eqnarray}\label{3-29}
u^{\prime}(t)=c_H t^{H-\frac{1}{2}}\left(I_{0+}^{H-\frac{1}{2}}\left(\psi^{-1} \dot{u}\right)\right)(t)=\frac{c_H t^{H-\frac{1}{2}}}{\Gamma\left(H-\frac{1}{2}\right)} \int_0^t(t-s)^{H-\frac{3}{2}} s^{\frac{1}{2}-H} \dot{u}(s) d s
\end{eqnarray}
with $\psi(u)=u^{H-\frac{1}{2}}$  {\cite[Lemma 2.3]{2020Budhiraja}}.
According to a direct   computation, we see that
\begin{eqnarray*}\label{3-30}
	|\tilde z^{(j)}_t-\tilde z_t| &\le&|\int_0^t (Df( x_s)\tilde z^{(j)}_s -Df(x_s)\tilde z_s)d s |+|\int_0^t \sigma(  x_s) (d u_s^{(j)}-d u_s) |\cr
	&=:& G_1+G_2.
\end{eqnarray*}
Then, by the Assumption (\textbf{A3}) and fact that $x$ is globally bounded, we have 
\begin{eqnarray*}\label{3-31}
	G_1 =|\int_0^t (Df( x_s)\tilde z^{(j)}_s -Df(x_s)\tilde z_s)d s | \le C\int_0^t |\tilde z_s^{(j)} -\tilde z_s|d s \le C\sup_{ t\in[0,T]}|\tilde z_t^{(j)} -\tilde z_t| .
\end{eqnarray*}
Since  $\{\tilde z^{(j)}\}_{ j\ge 1}$ converges to $\tilde z$ in the uniform topology, we have that
$G_1\to0$ as $j \to \infty$.

Then, it proceeds to estimate the second term $G_2$. By Fubini's theorem, we have 
\begin{eqnarray}\label{3-35}
G_2 &=&{\big|\int_0^t \sigma( x_s)(d u^{(j)}_s-d u_s) \big|} \cr
&=&{\bigg|\int_0^t \sigma( x_s) s^{H-\frac{1}{2}}\big[\int_0^s(s-r)^{H-\frac{3}{2}} r^{\frac{1}{2}-H}(\dot u^{(j)}_r-\dot u_r)dr\big] ds \bigg|} \cr
&=&{\bigg|\int_0^T \mathbf{1}_{[0,t]}(r) r^{\frac{1}{2}-H}\big[\int_r^t(s-r)^{H-\frac{3}{2}} s^{H-\frac{1}{2}}\sigma( x_s)ds\big]  (\dot u^{(j)}_r-\dot u_r)dr \bigg|}.
\end{eqnarray}
 {Let 
$ H_t(r)=\big[\int_r^t(s-r)^{H-\frac{3}{2}} s^{H-\frac{1}{2}}\sigma( x_s)ds\big] $}. Then,   according to H\"older's inequality and fact that  $x$ is globally bounded, we have
$|H_t(r)|\le C(1+\| x\|_{\infty})\le C$. Note that the last term in \eqref{3-35} is the $L_2$-inner product of $\dot u^{(j)}-\dot u$ and $r\mapsto \mathbf{1}_{[0,t]}(r) r^{\frac{1}{2}-H}H_t(r)$.
Since $\dot u^{(j)}\rightarrow\dot u$ as $j\rightarrow\infty$ in the weak topology of $L^2$,  $G_2$ converges to 0 as $j\rightarrow\infty$ for every fixed $t$.

Hence, $\{\tilde z^{(j)}\}_{ j\ge 1}$ has a unique limit point $\tilde z$, which implies $\lim_{j \to \infty}\tilde z^{(j)}=\tilde z$ in $ C^{\alpha}([0,T],\mathbb{R}^m)$.
The proof of Step 1 is completed.


\textbf{Step 2}.
In this step,  $N\in\mathbb{N}$ and $\varepsilon\in(0,1]$. We will let $\varepsilon \to 0$. The constant $C$ is  positive and may depend only on $p$ and $N$. It may vary from line to line.

Assume  $u^{\varepsilon}\in \mathcal{A}^{N}_b$    converges  in distribution to $u$ as $\varepsilon \to 0$. We  rewrite  the controlled deviation component (\ref{3.3}) as follows,
\[
\tilde {z}^{\varepsilon}:=\mathcal{G}^{\varepsilon}\bigg(\frac{B^H}{h(\varepsilon)}  +u^{\varepsilon}\bigg).
\]
In the following proof, we will show  that  as $\varepsilon\rightarrow 0$, $\tilde {z}^{\varepsilon}$  converges to $\tilde {z}$ in $ C^{\alpha}([0,T],\mathbb{R}^m)$ (in distribution), that is,
\begin{eqnarray} \label{3-34}
\mathcal{G}^{\varepsilon}\bigg( \frac{B^H}{h(\varepsilon)} +u^{\varepsilon}\bigg)\rightarrow\mathcal{G}^{0}(u )\quad \text{(in distribution)}.
\end{eqnarray}	
To this end, we construct the auxiliary process as follows,
\begin{eqnarray}\label{3-37}
d\hat z_t^\varepsilon=Df(x_t)\hat z_t^\varepsilon dt+\sigma(x_t)du_t^\varepsilon,\quad \hat z_0^\varepsilon=0.
\end{eqnarray}

Similar  to \lemref{lem3-2}, by using the fact that  $x\in C^1([0,T],\mathbb{R}^m)$, there exists $C>0$ such that, for every $u^{\varepsilon}\in \mathcal{A}_b^N$ and $p\geq 1$,
\begin{equation}\label{3-42}
\mathbb{E}\big[\sup_{ t\in[0,T]}|\hat z^{\varepsilon}_t|^p\big] \leq C.
\end{equation}
Then, based on  (\ref{3-42}) and take a similar manner to \lemref{lem3-3}, there exists   $C>0$ such that, for every $u^{\varepsilon}\in \mathcal{A}_b^N$ and $p\geq 1$,
\begin{equation}\label{z}
\mathbb{E}\big[\|\hat z^{\varepsilon}\|_{\alpha, \infty}^p\big] \leq C.
\end{equation}

Next, we will   estimate 
\begin{eqnarray*}\label{3-38}
	\mathbb{E}\big[\sup _{t \in[0, T]}\|\tilde z^{\varepsilon}-\hat{z}^{\varepsilon}\|_{\alpha,[0,t]}^2\big].
\end{eqnarray*}
Let $R>0$ be large enough and define the stopping times $\tau_R:=\inf \left\{t \in[0,T]:\|B^H\|_{1-\alpha, \infty, t} \geq R\right\}$.
Firstly, we have
\begin{eqnarray}\label{b3-40}
\mathbb{E}\Big[\sup _{t \in[0, T]}\left\|\tilde {z}^{\varepsilon}-\hat{z}^{\varepsilon}\right\|_{\alpha,[0,t]}^2\Big] &\leq&  \mathbb{E}\left[\sup _{t \in[0, T]}\left\|\tilde {z}^{\varepsilon}-\hat{z}^{\varepsilon}\right\|_{\alpha,[0,t]}^2 I_{\{\tau_R \le T\}}\right] \cr
&&+\mathbb{E}\left[\sup _{t \in[0, T]}\left\|\tilde {z}^{\varepsilon}-\hat{z}^{\varepsilon}\right\|_{\alpha,[0,t]}^2I_{\{\tau_R > T\}}\right].
\end{eqnarray}
For the first term on the right hand side of (\ref{b3-40}), by  H\"older's inequality, the definition of stopping time, \lemref{lem3-3} and \eqref{z}, we have
{\begin{eqnarray}\label{b3-41}
&&\mathbb{E}\Big[\sup _{t \in[0, T]}\left\|\tilde {z}^{\varepsilon}-\hat{z}^{\varepsilon}\right\|_{\alpha,[0,t]}^2 I_{\{\tau_R \le T\}}\Big] \cr
&\leq & \mathbb{E}\Big[\sup _{t \in[0, T]}\left\|\tilde {z}^{\varepsilon}-\hat{z}^{\varepsilon}\right\|_{\alpha,[0,t]}^4 \Big]^{1/2}\mathbb{P}(\tau_R \le T)^{1/2} \cr
&\leq & \mathbb{E}\Big[\sup _{t \in[0, T]}\left\|\tilde {z}^{\varepsilon}-\hat {z}^{\varepsilon}\right\|_{\alpha,[0,t]}^4 \Big]^{1/2}\times R^{-1}\sqrt{ \mathbb{E}\left[\left\|B^H\right\|_{1-\alpha, \infty, T}^2\right]} \cr
&\leq& C R^{-1}\sqrt{ \mathbb{E}\left[\left\|B^H\right\|_{1-\alpha, \infty, T}^2\right]}.
\end{eqnarray}}

Now we compute the second term on the right hand side of  (\ref{b3-40}).
Define $D:=\{\|B^H\|_{1-\alpha, \infty, T} \leq R\}$. For $\lambda\geq 1$, we will estimate the following 
\begin{eqnarray*}\label{b3-39}
	\mathbf{K}:=\mathbb{E}\left[\sup _{t \in[0, T]} e^{-\lambda t}\|\tilde z^{\varepsilon}-\hat{z}^{\varepsilon}\|_{\alpha,[0,t]}^2 \mathbf{1}_D\right].
\end{eqnarray*}
It is easy to see that,
\begin{eqnarray*}\label{3-40}
	\mathbf{K} &\leq & C \mathbb{E}\left[\sup _{t \in[0, T]}e^{-\lambda t}\left\|\int_0^\cdot[\frac{1}{\sqrt{\varepsilon}h(\varepsilon)}\big(f( \tilde x_s^{\varepsilon})-f(  x_s)\big)-Df(x_s)\hat{z}_s^\varepsilon] d s\right\|_{\alpha,[0,t]}^2 \mathbf{1}_D\right] \cr
	&&+C \mathbb{E}\left[\sup _{t \in[0, T]}e^{-\lambda t}\left\|\int_0^\cdot\big(\sigma( \tilde x_s^{\varepsilon})-\sigma(  x_{s})\big) d u_s^{\varepsilon}\right\|_{\alpha,[0,t]}^2 \mathbf{1}_D\right] \cr
	&&+C \mathbb{E}\left[\frac{1}{h^2(\varepsilon)}\sup _{t \in[0, T]}e^{-\lambda t}\left\|\int_0^\cdot\sigma_1( \tilde x_s^{\varepsilon}) d B_s^H\right\|_{\alpha,[0,t]}^2 \mathbf{1}_D\right] \cr
	&=:& K_1+K_2+K_3.
\end{eqnarray*}
Before computing $K_i,i=1,2,3$, we give  prior estimates. For a measurable function $g:[0,T]\to \mathbb{R}^d$,  we set $A=:\|\int_{0}^{\cdot} gds\|_{{\alpha,[0,t]}}$. We can estimate $A$ as follows:
\begin{eqnarray}\label{3-39}
A & \leq&\left|\int_0^t g(s) d s\right|+\int_0^t(t-s)^{-1-\alpha} \int_s^t|g(r)| d r d s \cr
& \leq& t^\alpha \int_0^t(t-r)^{-\alpha}|g(r)| d r+C
\int_0^t(t-r)^{-\alpha}|g(r)| dr \cr
& \leq& C \int_0^t(t-r)^{-\alpha}|g(r)| d r.
\end{eqnarray}

Then, we rearrange $K_1$ and use   {Taylor's expansion}, 
\begin{eqnarray}\label{3-41}
K_1 &\leq & C \mathbb{E}\left[\sup _{t \in[0, T]}e^{-\lambda t}\left\|\int_0^\cdot[\frac{1}{\sqrt{\varepsilon}h(\varepsilon)}\big(f( \tilde x_s^{\varepsilon})-f(  x_s)\big)-Df(x_s)\hat{z}_s^\varepsilon] d s\right\|_{\alpha,[0,t]}^2 \mathbf{1}_D\right] \cr
&\leq & C\mathbb{E}\Big[\sup _{t \in[0, T]}\int_{0}^{t} e^{-\lambda(t-s)}(t-s)^{-2\alpha}e^{-\lambda s}\big|\frac{1}{\sqrt{\varepsilon}h(\varepsilon)}\big(f( \tilde x_s^{\varepsilon})-f(  x_s)\big)-Df(x_s)\hat{z}_s^\varepsilon\big|^2 \mathbf{1}_Dds\Big] \cr
&\le & C\mathbb{E}\Big[\sup _{t \in[0, T]}\int_{0}^{t} e^{-\lambda(t-s)}(t-s)^{-2\alpha}e^{-\lambda s}\big|\int_{0}^{1}\big(Df(x_s+\theta(\tilde x_s^{\varepsilon}-x_s) )-Df(x_s)\big)\tilde{z}^\varepsilon_s d\theta\big|^2 \mathbf{1}_Dds\Big] \cr
&&+C\mathbb{E}\Big[\sup _{t \in[0, T]}\int_{0}^{t} e^{-\lambda(t-s)}(t-s)^{-2\alpha}e^{-\lambda s}\big|Df(x_s)\big|^2\big|\tilde z_s^{\varepsilon}-\hat z_s^{\varepsilon}\big|^2 \mathbf{1}_Dds\Big] \cr
&\leq & C\mathbb{E}\Big[\sup _{t \in[0, T]}\int_{0}^{t} e^{-\lambda(t-s)}(t-s)^{-2\alpha}e^{-\lambda s}\big(\big|\tilde{z}_s^\varepsilon-\hat{z}_s^\varepsilon\big|^2+\big|R^\varepsilon_s\big|^2 \big)\mathbf{1}_Dds\Big].
\end{eqnarray}
Here,  $R_s^\varepsilon:=\int_{0}^{1}\big(Df(x_s+\theta(\tilde x_s^{\varepsilon}-x_s) )-Df(x_s)\big)\tilde{z}^\varepsilon_s d\theta$ which satisfies $\sup_{ t\in[0,T]}|R_t^\varepsilon|\le C\sqrt{\varepsilon}h(\varepsilon)\sup_{ t\in[0,T]}|\tilde z_t^\varepsilon|^2$,  due to (\textbf{A3}).
Then by condition that  $2\alpha<1$, \lemref{lem3-2} and \eqref{3-10}, we have
\begin{eqnarray}\label{3-46}
K_1 
& \leq& C\mathbb{E}\left[\sup _{t \in[0, T]} e^{-\lambda t}\big\|\tilde{z}^{\varepsilon}-\hat{z}^{\varepsilon}\big\|_{\alpha,[0,t]}^2  {\mathbf{1}_D}\right] \sup _{t \in[0, T]} \int_0^t e^{-\lambda(t-r)}(t-r)^{-2 \alpha} d r \cr
&&+ C\varepsilon h^2(\varepsilon) \int_0^Te^{-\lambda(t-s)}(t-s)^{-2\alpha}\mathbb{E}\big[\big|\tilde z_s^{\varepsilon}\big|^4\mathbf{1}_D\big]d s  \cr
&\leq &C \lambda^{2 \alpha-1} \mathbb{E}\big[\sup_{t \in[0, T]}e^{-\lambda t} \big\|\tilde{z}^{\varepsilon}-\hat{z}^{\varepsilon}\big\|_{\alpha,[0,t]}^2  {\mathbf{1}_D}\big]+C\varepsilon h^2(\varepsilon).
\end{eqnarray}

Before estimating $K_2$, we firstly compute the following: for a  measurable function $g:[0,T]\to \mathbb{R}^{m\times {d_1}}$, $${B}:=\left\|\int_0^\cdot g(r) d u^{\varepsilon}_r\right\|_{\alpha,[0,t]}.$$ 
By using \propref{p1} and \eqref{3-8}, we have
\begin{eqnarray}\label{3-43}
{B} &\leq& C \|u^{\varepsilon}\|_{1-\alpha,\infty,T} \int_0^t\left((t-r)^{-2 \alpha}+r^{-\alpha}\right)\left(|g(r)|+\int_0^r \frac{|g(r)-g(q)|}{(r-q)^{1+\alpha}} d q\right) d r\cr
&\leq& C \int_0^t\left((t-r)^{-2 \alpha}+r^{-\alpha}\right)\|g\|_{\alpha,[0,r]} d r.
\end{eqnarray}
Here, we used condition that  $\|u^{\varepsilon}\|_{1-\alpha,\infty,T}\le C\|u^{\varepsilon}\|_{\mathcal{H}^{H,d_1}}\le C\sqrt{2N} $.

With aid of  \cite[Lemma 7.1]{2002Rascanu}, we  estimate $K_2$ as follows
\begin{eqnarray*}\label{3-45}
	K_2 &\le &C \mathbb{E}\left[\sup _{t \in[0, T]}\left|\int_0^t e^{-\lambda t}\left[(t-r)^{-2 \alpha}+r^{-\alpha}\right]\big\|\sigma_1( \tilde {x}^{\varepsilon})-\sigma_1( {x})\big\|_{\alpha,[0,r]} \mathbf{1}_D d r\right|^2\right] \cr
	&\leq &C \mathbb{E}\left[\sup _{t \in[0, T]} \mid \int_0^t e^{-\lambda t}\left[(t-r)^{-2 \alpha}+r^{-\alpha}\right]\right.\cr
	&&\left.\times\left.\big(1+\Delta(\tilde{x}_r^{\varepsilon})+\Delta({x}_r)\big)\big\|\tilde {x}^{\varepsilon}-{x}\big\|_{\alpha,[0,r]} \mathbf{1}_D d r\right|^2\right],
\end{eqnarray*}
where $\Delta(\tilde{x}_r^{\varepsilon})=\int_0^r \frac{|\tilde{x}_r^{\varepsilon}-\tilde{x}_q^{\varepsilon}|}{(r-q)^{1+\alpha}} d q$ and $\Delta({x}_r)=\int_0^r \frac{|{x}_r-{x}_q|}{(r-q)^{1+\alpha}} d q$. By \lemref{lem3-1},  {$\Delta(\tilde{x}_r^{\varepsilon})\le \|\tilde x^{\varepsilon}\|_{\alpha, \infty}$}, which in turn is dominated by  $C(R,N)>0$ independent of $\varepsilon$. (Here, the condition $\|B^H\|_{1-\alpha, \infty, T} \leq R$ is used.) It is obvious that $\Delta({x}_r)\le C$ for some constant $C$.
Next, by the definition of $\tilde{z}^\varepsilon$,  \lemref{lem3-3},  \eqref{3-3} and \eqref{3-10}, we  see that
\begin{eqnarray}\label{3-47}
K_2 \le C  \varepsilon h^2(\varepsilon)\lambda^{2\alpha-1}\mathbb{E}\big[\sup_{t \in[0, T]}e^{-\lambda t} \big\|\tilde{z}^{\varepsilon}\big\|_{\alpha,[0,t]}^2 \big]
\le C  \varepsilon h^2(\varepsilon).
\end{eqnarray}

For the third term $K_3$, by  \lemref{lem3-1}, we have
\begin{eqnarray}\label{3-48}
K_3 &\le &\frac{1}{h^2(\varepsilon)} C \mathbb{E}\left[\sup _{t \in[0, T]}e^{-\lambda t}\left\|\int_0^\cdot\sigma_1( \tilde x_s^{\varepsilon}) d B_s^H\right\|_{\alpha,[0,t]}^2 \mathbf{1}_D\right]\le \frac{C}{h^2(\varepsilon)} 
\end{eqnarray}
for some constant $C=C(R,N)>0$ independent of $\varepsilon$.

Combining   {(\ref{3-46}), (\ref{3-47}) and  (\ref{3-48}), by taking large enough $\lambda$ such that $C\lambda^{2\alpha-1}<1$,} we have
\begin{eqnarray*}\label{3-49}
	\mathbb{E}\left[\sup _{t \in[0, T]} e^{-\lambda t}\|\tilde z^{\varepsilon}-\hat{z}^{\varepsilon}\|_{\alpha,[0,t]}^2 \mathbf{1}_D\right]\le  C\varepsilon h^2(\varepsilon)+\frac{C}{h^2(\varepsilon)}.
\end{eqnarray*}
Multiplying $e^{\lambda T}$, we obtain
\begin{eqnarray*}\label{3-50}
	\mathbb{E}\left[\sup _{t \in[0, T]} \|\tilde z_t^{\varepsilon}-\hat{z}_t^{\varepsilon}\|_{\alpha,[0,t]}^2 \mathbf{1}_D\right]\le  C\varepsilon h^2(\varepsilon)+\frac{C}{h^2(\varepsilon)}.
\end{eqnarray*}
With aid of  (\ref{b3-41}), we have
\begin{eqnarray*}\label{3-51'}
	\mathbb{E}\big[\sup _{t \in[0, T]}\|\tilde z^{\varepsilon}-\hat{z}^{\varepsilon}\|_{\alpha,[0,t]}^2\big] \leq  C\varepsilon h^2(\varepsilon)+\frac{C}{h^2(\varepsilon)}+C^{\prime}R^{-1}
\end{eqnarray*}
for certain constants $C=C(R,N)>0$ independent of $\varepsilon$, and $C^{\prime}$  independent of $\varepsilon,R$.
As $\varepsilon \to 0$,  we have $\varepsilon h^2(\varepsilon)\to 0$ and $\frac{1}{h^2(\varepsilon)}\to 0$. Take $\limsup_{\varepsilon \to 0}$  for each fixed large $R$, then  let $R \to\infty$. Then we have
\begin{eqnarray*}\label{3-52'}
	\lim_{\varepsilon \to 0}	\mathbb{E}\big[\sup _{t \in[0, T]}\|\tilde z^{\varepsilon}-\hat{z}^{\varepsilon}\|_{\alpha,[0,t]}^2\big] =0.
\end{eqnarray*}
According to the   continuous inclusion $C^{\alpha+\kappa} \subset W_0^{\alpha, \infty} \subset C^{\alpha-\kappa}$ for $0<\kappa<\alpha$, we obtain 
\begin{eqnarray*}\label{3-81}
	\lim_{\varepsilon \to 0}\mathbb{E}\big[\|\tilde z^{\varepsilon}-\hat{z}^{\varepsilon}\|_{\alpha\textrm{-hld}}^2\big]= 0 \quad \text{for every} \quad \alpha<1/2.
\end{eqnarray*}

According to step 1, if  $u^{\varepsilon}\in \mathcal{A}^{N}_b$  converges in distribution to $u$ as $\varepsilon \to 0$, then,  $\hat z^{\varepsilon}=\mathcal{G}^0 (u^{\varepsilon})$  converges in disribution to $\tilde z=\mathcal{G}^0 (u)$ in $ C^{\alpha}([0,T],\mathbb{R}^m)$ as $\varepsilon \to 0$. 
Then,  for any  {bounded  Lipschitz continuous function} $F:C^{\alpha}([0, T]; \mathbb{R}^m) \to\mathbb{R}$,  as $\varepsilon \to 0$, we have
\begin{eqnarray}\label{3-82}
|\mathbb{E}[F(\tilde z^{\varepsilon})]-\mathbb{E}[F(\tilde z)]|&\le& |\mathbb{E}[F(\tilde z^{\varepsilon})]-\mathbb{E}[F(\hat z^{\varepsilon})]|+|\mathbb{E}[F(\hat z^{\varepsilon})]-\mathbb{E}[F(\tilde z)]|\cr
&\le& \|F\|_\textrm{Lip}\mathbb{E}[\|\tilde z^{\varepsilon}-\hat z^{\varepsilon}\|_{\alpha\textrm{-hld}}^2]^{\frac{1}{2}}+|\mathbb{E}[F(\hat z^{\varepsilon})]-\mathbb{E}[F(\tilde z)]|.
\end{eqnarray} 
where $\|F\|_\textrm{Lip}$ is the Lipschitz constant of function $F$. As $\varepsilon \to 0$, the right hand side of \eqref{3-82} tends to $0$. Then, by  {Portmanteau}'s theorem \cite[Theorem 13.16]{2020Klenke}, we obtain the desired  convergence in distribution (\ref{3-34}). The proof of Step 2 is finished.

\textbf{Step 3}. By step 1 and step 2, we can show the goodness of $I$, the Laplace upper bound
\begin{eqnarray}\label{5-5}
\limsup _{\varepsilon \rightarrow 0} -\frac{1}{h^2(\varepsilon)}  \log \mathbb{E}\left[e^{-F(z^{\varepsilon}) h^2{(\varepsilon)}}\right] \leq\inf _{\phi \in C^{\alpha}([0,T],\mathbb{R}^m)}\left[F(\phi)+I(\phi)\right]
\end{eqnarray}
and the lower bound 
\begin{eqnarray}\label{5-6}
\liminf _{\varepsilon \rightarrow 0} -\frac{1}{h^2(\varepsilon)} \log \mathbb{E}\left[e^{-F(z^{\varepsilon}) h^2{(\varepsilon)}}\right] \geq\inf _{\phi \in C^{\alpha}([0,T],\mathbb{R}^m)}\left[F(\phi)+I(\phi)\right]
\end{eqnarray}
for  every bounded and continuous     function $F: C^{\alpha}([0,T],\mathbb{R}^m)\to \mathbb{R}$
in the same manner as in \cite[Theorem 3.1]{2023Inahama}.
Since the Laplace principle \eqref{5-5}--\eqref{5-6} and the large deviation principle are equivalent  under the goodness of $I$, the theorem is proved.  \qed


\section{MDP for the two-time scale system}	
 {In this section, we 
 aim  to show that the slow variable $x^\varepsilon$ of the two-time scale SDEs (\ref{1}) satisfies an MDP with speed $b(\varepsilon)=1/h^2(\varepsilon)$.} Firstly, we give assumptions and the  statement of our main result.
\subsection{Preliminaries, Assumptions and  Main Result}
To ensure the existence and uniqueness of solutions to the two-time scale SDE (\ref{1}), we impose the following conditions:
\begin{itemize}
	\item[(\textbf{A1'}).] Assume that the function $\sigma_1$ is of  $C^1$ and there exists a  constant $L> 0$ such that for  any $ x_1 , x_2\in \mathbb{R}^{m}$, 
	$$\left|D \sigma_1\left( x_1\right)\right| \leq L,	\qquad
	\left|D \sigma_1\left( x_1\right)-D \sigma_1\left( x_2\right)\right| \leq L\left|x_1-x_2\right|$$
	hold. Here, $D$ is the standard gradient operator on $\mathbb{R}^{m}$.
	\item[(\textbf{A2'}).] Assume that there exists a constant $L> 0$ such that, for  any $ (x_1,y_1) $,  $ (x_2,y_2)\in \mathbb{R}^{m} \times\mathbb{R}^{n}$, 
	\begin{equation*}
	\begin{aligned}
	&\left|f_1\left( x_1, y_1\right)-f_1\left( x_2, y_2\right)\right| +
	\left|f_2\left(x_1, y_1\right)-f_2\left(x_2, y_2\right)\right|
	\\
	&\qquad \qquad 
	+\left|\sigma_2\left(x_1, y_1\right)-\sigma_2\left(x_2, y_2\right)\right|  \leq L\left(\left|x_1-x_2\right|+\left|y_1-y_2\right|\right), 
	\end{aligned}
	\end{equation*}	
	and that, for any $x_1\in \mathbb{R}^m$,
	 {	\begin{equation*}
	\sup_{y_1\in\mathbb{R}^{n}}(\left|f_1\left(x_1, y_1\right)\right| +\left|\sigma_2\left(x_1, y_1\right)\right| ) \leq L\left(1+\left|x_1\right|\right)
	\end{equation*}}
	hold.
\end{itemize}	
Clearly, (\textbf{A2'}) implies that there exists a constant $L^{\prime}>0$ such that for any $(x,y)\in \mathbb{R}^m \times \mathbb{R}^n$, 	$$
\left|f_2\left(x, y\right)\right| \leq L^{\prime}\left(1+\left|x\right|+\left|y\right|\right)$$
holds.
Under Assumptions (\textbf{A1'}) and (\textbf{A2'}),   {we can deduce}   from \cite[Theorem 2.2]{Guerra} that the original two-time scale system (\ref{1}) admits a unique (pathwise) solution $(x^{\varepsilon},y^{\varepsilon})\in C^\alpha([0,T], \mathbb{R}^m) \times C([0,T], \mathbb{R}^n)$ with any $\alpha \in (1-H, 1/2)$. 

In order to study 
an MDP for the system (\ref{1}),
we further impose: 	
\begin{itemize}
	\item[(\textbf{A3'}).]   Assume that there exist positive constants $C>0$ and $\beta_i >0 \,(i=1,2)$ such that for any  $ (x,y_1),(x,y_2)\in \mathbb{R}^{m} \times\mathbb{R}^{n}$
	\begin{equation*}
	\begin{aligned}
	2\left\langle y_1-y_2, f_2\left(x, y_1\right)-f_2\left(x, y_2\right)\right\rangle+\left|\sigma_2\left(x, y_1\right)-\sigma_2\left(x, y_2\right)\right|^2 
	&\leq-\beta_1\left|y_1-y_2\right|^2, \\
	2\left\langle y_1, f_2\left(x, y_1\right)\right\rangle+\left|\sigma_2\left(x, y_1\right)\right|^2 & \leq-\beta_2\left|y_1\right|^2+C|x|^2+C
	\end{aligned}
	\end{equation*}
	hold.
\end{itemize}		

Assumption  (\textbf{A3'})  ensures that the solution to the following SDE with frozen ${x}\in\mathbb{R}^{m}$
\[
d\tilde{y}_t = f_{2}({x}, \tilde{y}_t) dt +\sigma_2( {x}, \tilde{y}_t)dW_{t}
\]
has a unique invariant probability measure $\mu_{{x}}$ for every $x$. See \cite[Theorem 6.3.2]{DaPratoZabczyk}. 

From \cite{2023Pei}, it shows that as $\varepsilon \to 0$, the slow component strongly converges to a solution to an  ODE  as follows,  
\begin{equation}\label{4-2}
\left\{\begin{array}{l}
d \bar{x}_t=\bar{f}_1\left(\bar{x}_t\right) d t \\
\bar{x}_0=x_0 \in \mathbb{R}^m,
\end{array}\right.
\end{equation}
where $\bar{f}_1(x)=\int_{\mathbb{R}^{n}} f_1(x, y) d \mu_x(y)$ and $\mu_x$ is a unique invariant probability measure of the fast component with ``frozen"  $x \in \mathbb{R}^m$. It is not too difficult to see that $\bar{f}_1(x)$ also satisfies Lipschitz continuous and linear growth condition. Then, we can obtain that $\{\bar {x}_t\}_{t\in[0,T]}\in C([0,T],\mathbb{R}^m)$ is globally bounded.
\begin{itemize}
	\item[(\textbf{A4'}).]   Assume that $\bar f_1\in C^1$ and there exists positive constant $L>0$ such that for any  $ x_1,x_2\in \mathbb{R}^{m}$
	\begin{equation*}
	\begin{aligned}
	&\left|D\bar f_1\left( x_1\right)-D\bar f_1\left( x_2\right)\right|  \leq L\left|x_1-x_2\right|, 
	\end{aligned}
	\end{equation*}	
	hold.
\end{itemize}	
 {Now we give an example that assure (\textbf{A4'}) hold, in which  (\textbf{A1'})--(\textbf{A3'}) are implicitly assumed.
\begin{exm}
Let $a:=\sigma_{2}\sigma_{2}^{T}$. For $N\ge 1$, denote by $C^N_b$ the space of  bounded $C^N$-functions  whose derivatives up to order $N$ are also bounded. We impose the following assumptions: 
\begin{itemize}
	\item [(\textbf{H1}).]Assume that for every $x\in \mathbb{R}^m$, $f_1(x,\cdot), f_2(x,\cdot), a(x,\cdot)\in C^2_b$.
	\item [(\textbf{H2}).] Assume that the coefficient $a$ is nondegenerate in $y$ uniformly with respect to  $x$, that is, there exists $\hat L>1$ such that for any $x\in \mathbb{R}^m$, 
	$$\hat L^{-1}|\xi|^2\le |a(x,y)\xi|^2\le \hat L|\xi|^2,\quad \xi \in \mathbb{R}^n.$$
	\item [(\textbf{H3}).] Assume that 
	$\lim_{|y| \to  \infty}\sup_{ x}\langle y,f_2(x,y)\rangle  =-\infty$.
	\end{itemize}
If (\textbf{A2'}) and (\textbf{H1})--(\textbf{H3}) hold, thanks to the result \cite[Theorem 2.1]{2021Rockner}, we have that
for every  $F(x,y):=f_1(x,y)-\bar f_1(x)$, there exists a solution $u(x,\cdot)$ satisfying that $u(x,\cdot) \in C^3_b$ for every $x\in \mathbb{R}^m$ and  $u(\cdot,y)\in C_b^1$ for every $y\in \mathbb{R}^n$, to the following Poisson equation in $ \mathbb{R}^n$:
$$\mathcal{L}(x,y)u(x,y)=F(x,y),\quad y\in \mathbb{R}^n$$
	where
	$$\mathcal{L}(x,y): = \sum_{i,j=1}^{n}a^{i,j}(x,y)\frac{\partial^2}{\partial y_i \partial y_j}+\sum_{i}^{n}f_2^{i}(x,y)\frac{\partial}{\partial y_i}.$$
Then, according to \cite[Lemma 3.2]{2021Rockner}, we obtain that  $\bar f_1\in C^2_b$, so 
 (\textbf{A4'}) is satisfied.
\end{exm}}
Then we give two other examples that satisfy (\textbf{A4'}). Moreover, we assume:
\begin{itemize}
	\item 	[($\spadesuit$).]   Assume that $ f_1\in C^1$ and there exists positive constant $L>0$ such that for any  $ x_1,x_2\in \mathbb{R}^{m}$ and $ y\in \mathbb{R}^{n}$
	\begin{equation*}
	\begin{aligned}
	&\left|D_{x} f_1\left( x_1,y\right)-D_{x} f_1\left(x_2,y\right)\right|  \leq L\left|x_1-x_2\right| 
	\end{aligned}
	\end{equation*}	
	hold, where $D_{x} f_1\left( x,y\right)$ is the standard gradient operator with respect to the $x$-variable.
\end{itemize}
\begin{exm}
	Consider the following two-time scale system, where  the fast component is independent of slow component:
	\begin{eqnarray}\label{11}
	\left
	\{
	\begin{array}{ll}
	dx^{\varepsilon}_t = f_{1}(x^{\varepsilon}_t, y^{\varepsilon}_t)dt + \sqrt \varepsilon  \sigma_{1}( x^{\varepsilon}_t)dB^H_{t},\\
	dy^{\varepsilon}_t = \frac{1}{\varepsilon}f_{2}( y^{\varepsilon}_t)dt +\frac{1}{\sqrt \varepsilon} \sigma_2(  y^{\varepsilon}_t)dW_{t}
	\end{array}
	\right.
	\end{eqnarray}
	with the (non-random) initial condition 
	$(x^{\varepsilon}_0, y^{\varepsilon}_0)=(x_0, y_0)\in \mathbb{R}^{m}\times \mathbb{R}^{n}$. Then, (\textbf{A4'}) can be reduced to ($\spadesuit$).
\end{exm}
\begin{exm}
	Consider the following two-time scale system:
	\begin{eqnarray}\label{12}
	\left
	\{
	\begin{array}{ll}
	dx^{\varepsilon}_t = f_{1}(x^{\varepsilon}_t, y^{\varepsilon}_t)dt + \sqrt \varepsilon  \sigma_{1}( x^{\varepsilon}_t)dB^H_{t},\\
	dy^{\varepsilon}_t = \frac{1}{\varepsilon}f_{2}( x^{\varepsilon}_t,y^{\varepsilon}_t)dt +\frac{1}{\sqrt \varepsilon}dW_{t}
	\end{array}
	\right.
	\end{eqnarray}
	with the (non-random) initial condition 
	$(x^{\varepsilon}_0, y^{\varepsilon}_0)=(x_0, y_0)\in \mathbb{R}^{m}\times \mathbb{R}^{n}$. Here, $W =(W_t)_{t \in [0,T]}$ are  $d_2$-dimensional standard Brownian motion (BM) independent of FBM $B^H$. And we set $f_2=-D_{y} U(x,y)$ with $U(x,y)=|y|^2 a(x)$, where $a:\mathbb{R}^m\mapsto[\kappa,\infty)$ for some $\kappa>0$. Moreover, we assume $a$ is a bounded $C^2$-function with bounded first and second derivitives. Then, we can obtain the solution to the following SDE with frozen ${x}\in\mathbb{R}^{m}$
	\[
	d\tilde{y}_t =-2 a(x)\tilde{y}_tdt +dW_{t}
	\]
	has a unique invariant probability measure $\mu_{{x}}(dy)=\frac{{(4a(x))}^{n/2}\exp{[-2a(x)|y|^2]}}{{(2\pi)}^{n/2}}dy=\frac{{(2a(x))}^{n/2}\exp{[-2a(x)|y|^2]}}{{\pi}^{n/2}}dy$ for every $x$. 
	Then, the averaged coefficient is rewritten as following:
	\begin{eqnarray*}
		\bar f_1 (x)= \frac{{(2a(x))}^{n/2}}{{\pi}^{n/2}}\int_{\mathbb{R}^n}f_1(x,y)\exp{[-2a(x)|y|^2]}dy.
	\end{eqnarray*}
	By  the  condition that $a\geq\kappa$, $a$ is a bounded $C^2$-function and   direct computations, we will verify that if ($\spadesuit$) is assumed, then $\bar f_1 (x)$ satisfies (\textbf{A4'}). 
	
	Now we will see the Lipschitz continuity of $\bar f_1$. First, note that $f_1$ and $D_x f_1$ are both bounded. Since $a$ is bounded and $a\geq\kappa$, $x\mapsto {(a(x))}^{n/2}$ is clearly  Lipschitz. So it is sufficient to prove that $k(x)=\int_{\mathbb{R}^n}f_1(x,y)\exp{[-2a(x)|y|^2]}dy$ is Lipschitz continuous in $x$. Note that $k$ is bounded. Then, we have
	\begin{eqnarray}\label{exam2}
	D_xk(x)&=&\int_{\mathbb{R}^n}D_xf_1(x,y)\exp{[-2a(x)|y|^2]}dy\cr
	&&-2\int_{\mathbb{R}^n}f_1(x,y)|y|^2 Da(x)\exp{[-2a(x)|y|^2]}dy.
	\end{eqnarray}
	Here, since $\exp{[-2a(x)|y|^2]}\leq \exp{[-2\kappa|y|^2]}$, the differentiation and integral sign is easily verified. 
	From mean value theorem, we have  for $u,u^\prime\geq 0$, 
 {	\begin{eqnarray*}
		|e^{-u}-e^{-u^\prime}|\leq e^{-(u \wedge u^\prime)}|u-u^\prime|.
	\end{eqnarray*}}
	Then, we have for all $x,x^\prime \in \mathbb{R}^m$,
	\begin{eqnarray}\label{exam3}
	|e^{-2a(x)|y|^2}-e^{-2a(\tilde x)|y|^2}|&\leq& e^{-2|y|^2(a( x) \wedge a(\tilde x))}|a( x)-a(\tilde x)|\cr
	&\leq& e^{-2\kappa|y|^2}|a( x)-a(\tilde x)|\cr
	&\leq& e^{-2\kappa|y|^2}\|a\|_{\infty}| x-\tilde x|.
	\end{eqnarray}
	It is easy to see from \eqref{exam2} and \eqref{exam3} that $D_xk$ is Lipschitz continuous. Here, we also use ($\spadesuit$) and (\textbf{A2'}) as well as the boundness of $a, Da$ and $D^2 a$.
\end{exm}

\bigskip

Next, we define the	deviation component $z^\varepsilon$   as follows:
\begin{eqnarray}\label{4-0}
z_t^{\varepsilon} =\frac{x^\varepsilon_t-\bar x_t}{\sqrt{\varepsilon}h(\varepsilon)}
\end{eqnarray}
for $t\in[0,T]$.
It satisfies the following SDE:
\begin{eqnarray}\label{4-1}
dz^{\varepsilon}_t = \frac{1}{\sqrt{\varepsilon}h(\varepsilon)}\big(f(x^{\varepsilon }_t,y^{\varepsilon }_t)-f(\bar x_t)\big)dt + \frac{1}{h(\varepsilon)}  \sigma( x^{\varepsilon}_t)dB^H_{t},\quad z^\varepsilon_0=0.
\end{eqnarray}
Then there is a measurable map 
\[
\mathcal{G}^{\varepsilon} (\frac{1}{h(\varepsilon)}\bullet,\frac{1}{h(\varepsilon)}\star): C_0\left([0,T], \mathbb{R}^d\right) \rightarrow C^{\alpha}\left([0,T], \mathbb{R}^m\right)
\]
{with $\alpha\in (1-H,1/2)$} such that
$z^{\varepsilon}:=\mathcal{G}^{\varepsilon}\big(\frac{B^H}{h(\varepsilon)},\frac{W}{h(\varepsilon)}\big)$.

The skeleton equation is defined  as follows:
\begin{eqnarray}\label{3}
d\tilde{z}_t = D\bar{f}_1(\bar{x}_t)\tilde{z}_tdt + \sigma_{1}( \bar{x}_t)du_t,\quad \tilde{z}_0=0.
\end{eqnarray}	
Assumption   (\textbf{A4'}) implies that there exists a constant $K$ independent of $t$ and $z$ such that $|D\bar f(\bar x_t)z|\le K|z|$ for any $t\in[0,T]$ and $z$. 
Hence,  according to  \cite[Proposition 3.6]{2020Budhiraja},  there exists a unique solution $\tilde {z} \in  W_0^{\alpha, \infty}([0,T], \mathbb{R}^{m})$  to the skeleton equation (\ref{3}) for any $u \in S_{N}$. Moreover, we  obtain
$$\|\tilde{z}\|_{1-\alpha}\le c_N,$$
where $c_N$ independent of $u\in  S_N$. Because of  the  compact embedding  $C^{1-\alpha}([0,T],\mathbb{R}^{m}) \subset C^{\alpha}([0,T],\mathbb{R}^{m})$ for $1-\alpha>\alpha$, we   could define a map
\[
\mathcal{G}^{0}: \mathcal{H} \rightarrow  C^{\alpha}\left([0,T], \mathbb{R}^m\right)
\]
by  $\tilde{z}=\mathcal{G}^{0}(u, v)$. In other words, this is the solution map of the  skeleton equation \eqref{3}. Here, the solution map $\mathcal{G}^{0}(u, v)$ is independent	of $v$, i.e. $\mathcal{G}^{0}(u, v)=\mathcal{G}^{0}(u, 0)$.

Now, we provide a precise statement of our  second main theorem. Recall that $h(\varepsilon)=\varepsilon^{-\frac{\theta}{2}}$ for $\theta\in(1/2,1)$.
\begin{thm}\label{thm}
	Let $H\in(1/2,1)$  and $\alpha\in (1-H,1/2)$. Assume (\textbf{A1'})--(\textbf{A4'}).
	Then, as $\varepsilon \to 0$,
	the slow component $\{x^{\varepsilon}\}_{\varepsilon\in(0,1]}$ in (\ref{1}) satisfies an MDP with speed $b(\varepsilon)=1/h^2(\varepsilon)$ on 
	$C^{\alpha}\left([0,T], \mathbb{R}^m\right)$ with the good rate function $I: C^{\alpha}\left([0,T], \mathbb{R}^m\right)\rightarrow [0, \infty]$   defined by
	\begin{eqnarray*}\label{rate}
		I(\xi) &=& {
			\inf\Big\{\frac{1}{2}\|u\|^2_{\mathcal{H}^{H,d_1}}~:~{ u\in \mathcal{H}^{H,d_1} 
				\quad\text{such that} \quad\xi =\mathcal{G}^{0}(u, 0)}\Big\} 
		}, \qquad 
		\xi\in C^{\alpha}\left([0,T], \mathbb{R}^m\right).
	\end{eqnarray*}
\end{thm}	
Note that $I(\xi) 
= \inf\Big\{\frac{1}{2}\|(u,v)\|^2_{\mathcal{H}} :{( u,v)\in \mathcal{H}\,\,\text{such that} \,\,\xi =\mathcal{G}^{0}(u, v)}\Big\}
$ for $ 
\xi\in C^{\alpha}\left([0,T], \mathbb{R}^m\right)$.
By definition,  an MDP for $\{x^\varepsilon\}_{\varepsilon\in(0,1]}$ is  equivalent to an LDP for that $\{z^\varepsilon\}_{\varepsilon\in(0,1]}$, so we will prove the latter.

\subsection{Preliminary Lemmas}
In this subsection,  we show some prior estimates of the controlled two-time scale system that will be used in the proof of \thmref{thm}.
We continue to assume  that  $H\in(1/2,1)$ and  $\alpha\in (1-H,1/2)$.

In order to use the variational representation theorem,  we consider the following controlled two-time scale system associated to (\ref{1}) for a pair of control $(u^{\varepsilon}, v^{\varepsilon})\in \tilde{\mathcal{A}}^{b}$:  
\begin{eqnarray}\label{2}
\left
\{
\begin{array}{ll}
d\tilde {x}^{\varepsilon}_t =& f_{1}(\tilde {x}^{\varepsilon}_t, \tilde {y}^{\varepsilon}_t)dt + \sqrt{\varepsilon}h(\varepsilon)\sigma_{1}(\tilde {x}^{\varepsilon}_t)du^{\varepsilon}_t+ \sqrt \varepsilon  \sigma_{1}(\tilde {x}^{\varepsilon}_t)dB^H_{t},\\
d\tilde {y}^{\varepsilon}_t =& \frac{1}{\varepsilon} f_{2}( \tilde {x}^{\varepsilon}_t, \tilde {y}^{\varepsilon}_t)dt +\frac{h(\varepsilon)}{\sqrt{\varepsilon}}\sigma_2(\tilde {x}^{\varepsilon}_t, \tilde {y}^{\varepsilon}_t)dv^{\varepsilon}_t+ \frac{1}{\sqrt{\varepsilon}}\sigma_2(\tilde {x}^{\varepsilon}_t, \tilde {y}^{\varepsilon}_t)dW_{t}
\end{array}
\right.
\end{eqnarray}
with initial value $(\tilde x^{\varepsilon}_0, \tilde y^{\varepsilon}_0)=(x_0, y_0)\in \mathbb{R}^{m}\times \mathbb{R}^{n}$. It is easy to see that there exists a unique solution $(\tilde {x}^{\varepsilon}, \tilde {y}^{\varepsilon})$ to the controlled two-time scale  system (\ref{2}). 

Respectively, the controlled deviation component $\tilde z^\varepsilon$ satisfies the following  SDEs,
\begin{eqnarray}\label{4.3}
d\tilde {z}^{\varepsilon}_t =\frac{1}{\sqrt{\varepsilon}h(\varepsilon)}\big(f(\tilde x^{\varepsilon}_t,\tilde y^{\varepsilon}_t)-\bar f(\bar x_t)\big)dt + \sigma(\tilde {x}^{\varepsilon}_t)du^{\varepsilon}_t+ \frac{1 }{h(\varepsilon)} \sigma(\tilde {x}^{\varepsilon}_t)dB^H_{t},\quad \tilde z^{\varepsilon}_0=0.
\end{eqnarray}
Note that  $\tilde {z}^{\varepsilon}=\mathcal{G}^{\varepsilon} (\frac{B^H}{h(\varepsilon)} +u^{\varepsilon},\frac{W}{h(\varepsilon)} +v^{\varepsilon})$. 

\begin{lem}\label{lem1}
	Assume (\textbf{A1'})--(\textbf{A2'}), let  $p\ge1$ and  $N\in\mathbb{N}$. Then, there exists   $C>0$ such that for every $(u^{\varepsilon},v^{\varepsilon})\in \tilde{\mathcal{A}}_b^N$, we have
	\begin{equation*}
	\mathbb{E}\big[\|\tilde x^{\varepsilon}\|_{\alpha, \infty}^p\big] \leq C.
	\end{equation*}
	Here, $C$ is a positive constant which depends only on $p$ and $N$.
\end{lem}
\para{Proof}. Under  (\textbf{A1'})--(\textbf{A2'}),  the lemma can be proved in the same way as in \lemref{lem3-1}.
\qed
\begin{lem}\label{lem2}
	Assume (\textbf{A1'})--(\textbf{A2'}) and let $N\in\mathbb{N}$. Then, there exists $C>0$ such that, for every $(u^{\varepsilon},v^{\varepsilon})\in \tilde{\mathcal{A}}_b^N$ and every $(t,h)$ with $0 \leq t \leq t+h \leq T$, we have 
	$$\mathbb{E}\big[|\tilde x_{t+h}^{\varepsilon}-\tilde x_{t}^{\varepsilon}|^2\big] \leq C h^{2-2 \alpha}.$$
	Here, $C$ is a positive constant which depends only on $N$.
\end{lem}
\para{Proof}. Using the  condition that $\varepsilon,\sqrt{\varepsilon}h(\varepsilon) \in(0,1]$,  we can prove the desired inequality in the same way as in the \cite[Lemma 4.2]{2023Inahama}. \qed

\begin{lem}\label{lem3}
 {Let $N\in\mathbb{N}$.	Assume (\textbf{A1'})--(\textbf{A3'}) and  that there exists $\varepsilon_0\in(0,1]$ such that for all
		$\varepsilon\in (0,\varepsilon_0]$, $\sqrt{\varepsilon}h(\varepsilon)<\frac{\beta_2}{2}$}. Then, there exists $C>0$ such that for every $(u^{\varepsilon},v^{\varepsilon})\in \mathcal{A}_b^N$, we have
	\begin{eqnarray}\label{4.1}
	\int_{0}^{T} \mathbb{E}\big[|\tilde  y_t^{\varepsilon}|^2\big] dt\leq C.
	\end{eqnarray}
	Here, $C$ is a positive constant which depends only  on $N$.
\end{lem}
\para{Proof}. According to   It\^o's formula, we have
\begin{eqnarray}\label{4.1'}
{| \tilde {y}_{t}^{\varepsilon} |^2} &=& {| y_0 |^2} + \frac{2}{\varepsilon }\int_0^t {\langle \tilde {y}_{s}^{\varepsilon},f_2( \tilde {x}_{s}^{\varepsilon},\tilde {y}_{s}^{\varepsilon} )\rangle ds} 
+ \frac{2h(\varepsilon)}{{\sqrt { \varepsilon } }}\int_0^t {\langle \tilde {y}_{s}^{\varepsilon},\sigma_{2} ( {\tilde {x}_{s}^{\varepsilon},\tilde {y}_{s}^{\varepsilon}} ){\frac{dv_s^{\varepsilon}}{ds}  }\rangle ds}   \cr
&&+\frac{2}{\sqrt{\varepsilon} }\int_0^t {\langle \tilde {y}_{s}^{\varepsilon},\sigma_{2}( {\tilde {x}_{s}^{\varepsilon},\tilde {y}_{s}^{\varepsilon}} ) dW_s\rangle}+ \frac{1}{\varepsilon }\int_0^t |{\sigma_{2} }( {\tilde {x}_{s}^{\varepsilon},\tilde {y}_{s}^{\varepsilon}} )|^2ds.
\end{eqnarray}
In the same way as in  \cite[Lemma 4.3]{2023Inahama}, we  obtain that 
for every $(u^{\varepsilon},v^{\varepsilon})\in \mathcal{A}_b^N$, 
$\sup_{0\le s\le t}|\tilde  y_s^{\varepsilon}|$ has moments of all orders. 
Then,  by using \lemref{lem1}, we have that the fourth term in right hand side of (\ref{4.1'}) is a true martingale. In particular,   $\mathbb{E}[\int_0^t {\langle \tilde {y}_{s}^{\varepsilon},\sigma_{2}( {\tilde {x}_{s}^{\varepsilon},\tilde {y}_{s}^{\varepsilon}} ) dW_s\rangle}]=0$.
Taking expectation of \eqref{4.1'}, we obtain
\begin{eqnarray}\label{3.21}
\frac{d\mathbb{E}[{| \tilde {y}_{t}^{\varepsilon} |^2}]}{dt} &=& \frac{2}{\varepsilon }\mathbb{E} \big[{\langle \tilde {y}_{t}^{\varepsilon},f_2( \tilde {x}_{t}^{\varepsilon},\tilde {y}_{t}^{\varepsilon} )\rangle } \big]
+ \frac{2h(\varepsilon)}{{\sqrt { \varepsilon } }}\mathbb{E} \big[{\langle \tilde {y}_{t}^{\varepsilon},\sigma_{2} ( {\tilde {x}_{t}^{\varepsilon},\tilde {y}_{t}^{\varepsilon}} ){\frac{dv_t^{\varepsilon}}{dt}  }\rangle }\big]  \cr 
&&+ \frac{1}{\varepsilon }\mathbb{E} \big[|{\sigma_{2} }( {\tilde {x}_{t}^{\varepsilon},\tilde {y}_{t}^{\varepsilon}} )|^2\big].
\end{eqnarray} 
By  (\textbf{A2'}) and  (\textbf{A3'}), we have
\begin{eqnarray}\label{lemma3.12}
\frac{2}{\varepsilon }{\langle \tilde {y}_{t}^{\varepsilon},f_2( \tilde {x}_{t}^{\varepsilon},\tilde {y}_{t}^{\varepsilon} )\rangle } +\frac{1}{\varepsilon } |{\sigma_{2} }( {\tilde {x}_{t}^{\varepsilon},\tilde {y}_{t}^{\varepsilon}} )|^2\le- \frac{{ \beta_2 }}{\varepsilon }{{| \tilde {y}_{t}^{\varepsilon} |^2}}  + {\frac{C}{\varepsilon}|\tilde x^\varepsilon_t|^2}+ \frac{{C}}{\varepsilon }	
\end{eqnarray}
and
\begin{eqnarray}\label{3.22}
\begin{aligned}
\frac{2h(\varepsilon)}{{\sqrt { \varepsilon } }}{\langle \tilde {y}_{t}^{\varepsilon},\sigma_{2} ( {\tilde {x}_{t}^{\varepsilon},\tilde {y}_{t}^{\varepsilon}} ){\frac{dv_t^{\varepsilon}}{dt}  }\rangle }\le    {\frac{{2L^2h(\varepsilon)}}{{\sqrt { \varepsilon } }}}\big( 1 + | \tilde {x}_{t}^{\varepsilon} |^2 \big)| {\frac{dv_t^{\varepsilon}}{dt}  } |^2+   \frac{h(\varepsilon)}{{\sqrt { \varepsilon } }}| \tilde {y}_{t}^{\varepsilon} |^2 .
\end{aligned}
\end{eqnarray}
Thus, from (\ref{3.21})--(\ref{3.22}), we obtain
\begin{eqnarray*}\label{3.23}
	\begin{aligned}
		\frac{d\mathbb{E}[{| \tilde {y}_{t}^{\varepsilon} |^2}]}{dt} &\le  ({\frac{{ - \beta_2 }}{\varepsilon } }+\frac{{{h(\varepsilon) }}}{{\sqrt { \varepsilon } }}) \mathbb{E}[{| \tilde {y}_{t}^{\varepsilon} |^2}] + {\frac{C}{\varepsilon}|\tilde x^\varepsilon_t|^2} +    {\frac{{2L^2h(\varepsilon)}}{{\sqrt { \varepsilon } }}} \mathbb{E}[{{| \tilde {x}_{t}^{\varepsilon} |^2{| {\frac{dv_t^{\varepsilon}}{dt}  } |^2}}}]  +    {\frac{{2L^2h(\varepsilon)}}{{\sqrt { \varepsilon } }}} \mathbb{E}[{| {\frac{dv_t^{\varepsilon}}{dt}  } |^2}] + \frac{{ C}}{\varepsilon }\\
		&\le  {\frac{{ - \beta_2 }}{2\varepsilon } } \mathbb{E}[{| \tilde {y}_{t}^{\varepsilon} |^2}]  +    {\frac{{2L^2h(\varepsilon)}}{{\sqrt { \varepsilon } }}}  \mathbb{E}[{{| \tilde {x}_{t}^{\varepsilon} |^2{| {\frac{dv_t^{\varepsilon}}{dt}  } |^2}}}]  +    {\frac{{2L^2h(\varepsilon)}}{{\sqrt { \varepsilon } }}} \mathbb{E}[{| {\frac{dv_t^{\varepsilon}}{dt}  } |^2}] + \frac{{ C}}{\varepsilon }.
	\end{aligned}
\end{eqnarray*}
Moreover,  by using the comparison theorem for all $t$, we have  that
\begin{eqnarray}\label{3.231}
\begin{aligned}
\mathbb{E}[{| \tilde {y}_{t}^{\varepsilon} |^2}] &\le {|y_0|^2 }+   {\frac{C}{\varepsilon}\int_{0}^{t}  e^{-\frac{\beta_2}{2\varepsilon} (t-s)}|\tilde x^\varepsilon_s|^2ds}+  {\frac{{2L^2h(\varepsilon)}}{{\sqrt { \varepsilon } }}}\int_{0}^{t}  e^{-\frac{\beta_2}{2\varepsilon} (t-s)}{\mathbb{E}[{| \tilde {x}_{s}^{\varepsilon} |^2{| {\frac{dv_s^{\varepsilon}}{ds}  } |^2}}]} ds  \\
&\qquad+    {\frac{{2L^2h(\varepsilon)}}{{\sqrt { \varepsilon } }}}\int_{0}^{t}  e^{-\frac{\beta_2}{2\varepsilon} (t-s)} { \mathbb{E}[|{\frac{dv_s^{\varepsilon}}{ds}  } |^2]}ds + \frac{{ C}}{\varepsilon }\int_{0}^{t}  e^{-\frac{\beta_2}{2\varepsilon} (t-s)}ds.
\end{aligned}
\end{eqnarray}
Then, by integrating  both sides of (\ref{3.231}) and using Fubini's theorem, we  obtain
\begin{eqnarray*}\label{3.232}
	\begin{aligned}
		\int_{0}^{T}\mathbb{E}[{| \tilde {y}_{t}^{\varepsilon} |^2}]dt 
		&\le  {|y_0|^2 T}+ {\frac{C}{\varepsilon}\int_{0}^{T}\int_{0}^{t}  e^{-\frac{\beta_2}{2\varepsilon} (t-s)}|\tilde x^\varepsilon_s|^2ds}  +  {\frac{{2L^2h(\varepsilon)}}{{\sqrt { \varepsilon } }}}\int_{0}^{T}\int_{0}^{t}  e^{-\frac{\beta_2}{2\varepsilon} (t-s)}{\mathbb{E}[{| \tilde {x}_{s}^{\varepsilon} |^2{| {\frac{dv_s^{\varepsilon}}{ds}  } |^2}}]} dsdt  \\
		&\quad+   {\frac{{2L^2h(\varepsilon)}}{{\sqrt { \varepsilon } }}}\int_{0}^{T}\int_{0}^{t}  e^{-\frac{\beta_2}{2\varepsilon} (t-s)} { \mathbb{E}[|{\frac{dv_s^{\varepsilon}}{ds}  } |^2]}ds + \frac{{ C}}{\varepsilon }\int_{0}^{t}  e^{-\frac{\beta_2}{2\varepsilon} (t-s)}dsdt\\
		&\le  {|y_0|^2 T }+  {\frac{2CT}{\beta_2}\mathbb{E}\big[\sup_{0 \leq t \leq T}| \tilde {x}_{t}^{\varepsilon} |^2\big]}+   {\frac{{2L^2h(\varepsilon)}}{{\sqrt { \varepsilon } }}}\mathbb{E}\big[\sup_{0 \leq t \leq T}| \tilde {x}_{t}^{\varepsilon} |^2\int_{0}^{T}\int_{s}^{T}  e^{-\frac{\beta_2}{2\varepsilon} (t-s)} dt| {\frac{dv_s^{\varepsilon}}{ds}  } |^2ds\big]  \\
		&\quad+    {\frac{{2L^2h(\varepsilon)}}{{\sqrt { \varepsilon } }}}\int_{0}^{T}\int_{s}^{T}  e^{-\frac{\beta_2}{2\varepsilon} (t-s)}dt { \mathbb{E}[|{\frac{dv_s^{\varepsilon}}{ds}  } |^2]}ds + \frac{{ C}}{\varepsilon }\int_{0}^{t}  e^{-\frac{\beta_2}{2\varepsilon} (t-s)}ds\\
		&\le  {|y_0|^2 T}+ {\frac{2CT}{\beta_2}\mathbb{E}\big[\sup_{0 \leq t \leq T}| \tilde {x}_{t}^{\varepsilon} |^2\big]+   \frac{4L^2\sqrt { \varepsilon }  h(\varepsilon) }{\beta_2}\mathbb{E}\big[\sup_{0 \leq t \leq T}| \tilde {x}_{t}^{\varepsilon} |^2\int_{0}^{T}   | {\frac{dv_s^{\varepsilon}}{ds}  } |^2ds\big]}\cr
		&\quad  {  +   \frac{4L^2\sqrt { \varepsilon }  h(\varepsilon)}{\beta_2}\int_{0}^{T} { \mathbb{E}[|{\frac{dv_s^{\varepsilon}}{ds}  } |^2]}ds + C}.
	\end{aligned}
\end{eqnarray*}
By the fact that $(u^{\varepsilon}, v^{\varepsilon})\in \tilde{\mathcal{A}}_{b}^N$ and $\sqrt{\varepsilon}h(\varepsilon)\in(0,1]$, we have
\begin{eqnarray*}\label{3.24}
	\int_{0}^{T}\mathbb{E}[{| \tilde {y}_{t}^{\varepsilon} |^2}]dt &\le& {C}\mathbb{E}[\mathop {\sup }\limits_{t \in \left[ {0,T} \right]} {| \tilde {x}_{t}^{\varepsilon} |^2}]+C.
\end{eqnarray*}
Then, by using \lemref{lem1},  we obtain (\ref{4.1}). 
\qed

\subsection{Proof of Main Theorem(\thmref{thm})}\label{sec.5}
We will prove the LDP for $\{z^\varepsilon\}_{\varepsilon\in(0,1]}$ with speed $b(\varepsilon)=1/h^2(\varepsilon)$ on the  H\"older path space 	$C^{\alpha}\left([0,T], \mathbb{R}^m\right)$ with the good rate function $I$.

\textbf{Step 1}.  Everything in this step is deterministic.
Let $(u^{(j)}, v^{(j)})$ and $(u, v)$ belong to $S_N$ such that $(u^{(j)}, v^{(j)})\rightarrow(u, v)$ as $j\rightarrow\infty$ in the weak topology of $\mathcal{H}$. 
In this step,  we will  show 
\begin{eqnarray} \label{3.26}
\mathcal{G}^{0}( u^{(j)} ,v^{(j)} )\rightarrow\mathcal{G}^{0}(u,v) \label{step2}
\end{eqnarray}
in $C^{\alpha}([0,T],\mathbb{R}^m)$ as $j \to \infty$.

Due to  the  {continuous embedding}  
$\mathcal{H}^{H,d_1}\subset C^{1-\alpha}([0,T], \mathbb{R}^{d_1})$, 
we  obtain  that $\{u^{(j)}\}_{j\ge 1} \subset C^{1-\alpha}([0,T], \mathbb{R}^{d_1})$. Let $\{\tilde {z}^{(j)}_t\}_{j\ge 1}$ be a family of solutions to the skeleton equation (\ref{3}), that is,
\begin{eqnarray*} \label{3.25}
	d\tilde {z}^{(j)}_t =  D\bar{f}_1(  \bar{x}_t)\tilde {z}^{(j)}_tdt + \sigma_{1}( \bar{x}_t)du^{(j)}_t,\quad \tilde {z}^{(j)}_0=0.
\end{eqnarray*}	
Due to  Assumption (\textbf{A1'}), (\textbf{A4'}) and  \cite[Proposition 3.6]{2020Budhiraja}, there exists a unique solution $\tilde {z} \in  W_0^{\alpha, \infty}([0,T], \mathbb{R}^{m})$  to the skeleton equation (\ref{3}). Furthermore, there exists a positive constant $C:=C_N$ such that
\begin{equation*} \label{3.28}
\sup_{ j\geq 1}\|\tilde z^{(j)}\|_{1-\alpha} \leq C. 
\end{equation*}	

By Assumption (\textbf{A4'}) and taking same way as in Step 1 of \thmref{thm1},  {it is not too difficult to verify that there exists a subsequence of $\{\tilde z^{(j)}\}_{ j\ge 1}$  converges to the limit point $\tilde z$ in $C^{\alpha}([0,T], \mathbb{R}^{m})$ which  satisfies the following ODEs}:
\begin{eqnarray} \label{3.44}
d\tilde {z}_t =  D\bar{f}_1(  \bar{x}_t)\tilde z_tdt + \sigma_{1}( \bar{x}_t)du_t, \quad \tilde {z}_0=0.
\end{eqnarray}

\textbf{Step 2}. 
Let $N\in\mathbb{N}$, $\varepsilon\in(0,1]$ and $\Delta\in(0,T]$.
We will let $\varepsilon \to 0$ later.  In this step, the constant $C>0$ is independent of $\varepsilon,\Delta$ which may change from line to line.

Assume that $(u^{\varepsilon}, v^{\varepsilon})\in \tilde{\mathcal{A}}^{N}_b$  converges in distribution to $(u, v)$ as $\varepsilon \to 0$. We rewrite the controlled deviation component  in  (\ref{4.3}) as follows:
\[
\tilde {z}^{\varepsilon}:=\mathcal{G}^{\varepsilon}\bigg(\frac{B^H}{h(\varepsilon)}+u^{\varepsilon} , \frac{W}{h(\varepsilon)} +v^{\varepsilon}\bigg).
\]
We will prove that   $\tilde {z}^{\varepsilon}$  converges in distribution to $\tilde {z}$ in $ C^{\alpha}([0,T],\mathbb{R}^m)$ as $\varepsilon\rightarrow 0$, that is,
\begin{eqnarray} \label{step3}
\mathcal{G}^{\varepsilon}\bigg(\frac{B^H}{h(\varepsilon)}+u^{\varepsilon} , \frac{W}{h(\varepsilon)} +v^{\varepsilon}\bigg)\rightarrow\mathcal{G}^{0}(u , v) \quad\text{(in distribution)}.
\end{eqnarray}	
To do so, we first define the auxiliary process  
\begin{eqnarray*}
	d\bar {x}^{\varepsilon}_t =\bar f_{1}(\bar {x}^{\varepsilon}_{t})dt +  \sqrt{\varepsilon}\sigma_{1}(\bar {x}^{\varepsilon}_t)dB^{H}_t+ \sqrt{\varepsilon}h(\varepsilon)\sigma_{1}(\bar {x}^{\varepsilon}_t)du^{\varepsilon}_t,\quad \bar {x}^{\varepsilon}_0=x_0
\end{eqnarray*}
and   divide $\tilde z^\varepsilon=\hat z^\varepsilon+\bar z^\varepsilon$, where we set
 {\begin{eqnarray}\label{3.78}
	\hat  {z}^{\varepsilon}_t =\frac{\tilde{x}^{\varepsilon}_t-\bar {x}^{\varepsilon}_t}{\sqrt{\varepsilon}h(\varepsilon)},\qquad \bar  {z}^{\varepsilon}_t =\frac{\bar{x}^{\varepsilon}_t-\bar {x}_t}{\sqrt{\varepsilon}h(\varepsilon)}.
\end{eqnarray}}
Note that $\hat  {z}^{\varepsilon}_0=0$ and $\bar  {z}^{\varepsilon}_0=0$. In order to show \eqref{step3}, it is enough to  verify the following two statements \textbf{(a)} and \textbf{(b)}.
\begin{itemize}
	\item[\textbf{(a)}.] For any $\delta>0$,
	\begin{eqnarray}
	\lim_{\varepsilon \to  0}\mathbb{P}\big(\|\hat z^\varepsilon\|_{\alpha\textrm{-hld}}>\delta\big)=0.
	\end{eqnarray}
	\item[\textbf{(b)}.]As $\varepsilon \to 0$,
	\begin{eqnarray}\label{a}
	\bar z^\varepsilon \rightarrow\tilde z \quad (\text{in} \,\, C^{\alpha}([0,T],\mathbb{R}^m) \,\text{in distribution}).
	\end{eqnarray}
\end{itemize}

Firstly, we  verify  Statement \textbf{(a)}. It is equivalent to show that for any $\delta>0$,
\begin{eqnarray}\label{b}
\lim_{\varepsilon \to  0}\mathbb{P}\big(\|\tilde x^\varepsilon_t-\bar x^\varepsilon_t\|_{\alpha\textrm{-hld}}>\delta \sqrt{\varepsilon}h(\varepsilon)\big)=0
\end{eqnarray}
holds.
Now,	  we construct the auxiliary processes as follows:
\begin{eqnarray*}
	\left
	\{
	\begin{array}{ll}
		d\hat {x}^{\varepsilon}_t =& f_{1}(\tilde {x}^{\varepsilon}_{t(\Delta)}, \hat {y}^{\varepsilon}_t)dt + \sqrt{\varepsilon}h(\varepsilon)\sigma_{1}(\hat {x}^{\varepsilon}_t)du^{\varepsilon}_t+ \sqrt{\varepsilon}\sigma_{1}(\hat {x}^{\varepsilon}_t)dB^{H}_t,\\
		d\hat {y}^{\varepsilon}_t =&  \frac{1}{\varepsilon}f_{2}( \tilde {x}^{\varepsilon}_{t(\Delta)}, \hat {y}^{\varepsilon}_t)dt + \frac{1}{\sqrt \varepsilon}\sigma_2(\tilde {x}^{\varepsilon}_{t(\Delta)}, \hat {y}^{\varepsilon}_t)dW_{t}
	\end{array}
	\right.
\end{eqnarray*}
with the same initial condition as (\ref{2}), where $t(\Delta)=\left\lfloor\frac{t}{\Delta}\right\rfloor \Delta$ is the nearest breakpoint preceding $t$. 			
By essentially the same arguments as in  \lemref{lem1} and \lemref{lem3}, we have for every $p\geq1$,
\begin{eqnarray}\label{3-36}
\mathbb{E}\big[\|\hat x^{\varepsilon}\|_{\alpha, \infty}^p\big] \leq C
\end{eqnarray}
and 
\begin{eqnarray*}\label{3.37}
	\int_{0}^{T} \mathbb{E}\big[|\hat y_t^{\varepsilon}|^2\big] dt\leq C,
\end{eqnarray*}	
where the constant $C$ only depends on $p$ and $ N$.

Then, we will prove that  for any $\delta>0$,
\begin{eqnarray*}\label{3.38}
	\lim_{\varepsilon \to 0}\mathbb{P}\big(\frac{1}{\sqrt{\varepsilon}h(\varepsilon)}\|\tilde x^{\varepsilon}-\hat{x}^{\varepsilon}\|_{\alpha,\infty}\geq\delta\big)=0.
\end{eqnarray*}
Take $R>0$  large enough and define the stopping time $\tau_R:=\inf \left\{t \in[0,T]:\|B^H\|_{1-\alpha, \infty, t} \geq R\right\}$. For $\lambda>0$, we will estimate 
\begin{eqnarray*}\label{b3.39}
	\mathbf{J}:=\mathbb{E}\left[\sup _{t \in[0,T\wedge \tau_R]} e^{-\lambda t}\|\tilde x^{\varepsilon}-\hat{x}^{\varepsilon}\|_{\alpha,[0,t]}^2\right].
\end{eqnarray*}
It is easy to see that
\begin{eqnarray*}\label{3.40}
	\mathbf{J} &\leq & C \mathbb{E}\left[\sup _{t \in[0,T\wedge \tau_R]}e^{-\lambda t}\left\|\int_0^\cdot\big(f_1( \tilde x_s^{\varepsilon}, \tilde y_s^{\varepsilon})-f_1( \tilde x_{s(\Delta)}^{\varepsilon}, \hat{y}_s^{\varepsilon})\big) d s\right\|_{\alpha,[0,t]}^2 \right] \cr
	&&+C \mathbb{E}\left[\sup _{t \in[0,T\wedge \tau_R]}e^{-\lambda t}\left\|\sqrt{\varepsilon}h(\varepsilon)\int_0^\cdot\big(\sigma_1( \tilde x_s^{\varepsilon})-\sigma_1( \hat x_{s}^{\varepsilon})\big) d u_s^{\varepsilon}\right\|_{\alpha,[0,t]}^2 \right] \cr
	&&+C \mathbb{E}\left[\sup _{t \in[0,T\wedge \tau_R]}e^{-\lambda t}\left\|\sqrt{\varepsilon}\int_0^\cdot(\sigma_1( \tilde x_s^{\varepsilon})-\sigma_1( \hat x_s^{\varepsilon})) d B_s^H\right\|_{\alpha,[0,t]}^2 \right] \cr
	&=:& J_1+J_2+J_3.
\end{eqnarray*}

We will estimate $J_1$. By (\textbf{A2'}) and (\ref{3-39}), we have
\begin{eqnarray}\label{3.41}
J_1 &\leq & C \mathbb{E}\left[\sup _{t \in[0,T\wedge \tau_R]}\left\|\int_0^\cdot\big(f_1( \tilde x_s^{\varepsilon}, \tilde y_s^{\varepsilon})-f_1( \tilde x_{s(\Delta)}^{\varepsilon}, \tilde {y}_s^{\varepsilon})\big) d s\right\|_{\alpha,[0,t]}^2 \right] \cr
&&+C \mathbb{E}\left[\sup _{t \in[0,T\wedge \tau_R]}\left\|\int_0^\cdot\big(f_1( \tilde x_{s(\Delta)}^{\varepsilon}, \tilde {y}_s^{\varepsilon})-f_1( \tilde x_{s(\Delta)}^{\varepsilon}, \hat{y}_s^{\varepsilon})\big) d s\right\|_{\alpha,[0,t]}^2 \right] \cr
&\leq & C\Big[\sup _{t \in[0, T]}\int_{0}^{t} (t-s)^{-2\alpha}ds\Big] \Big\{\int_0^T\mathbb{E}\left[\big|f_1( \tilde x_s^{\varepsilon}, \tilde y_s^{\varepsilon})-f_1( \tilde x_{s(\Delta)}^{\varepsilon}, \tilde {y}_s^{\varepsilon})\big|^2 \right]d s \Big.\cr
&&\qquad\qquad\qquad\qquad\qquad\qquad\quad+\Big. \int_0^T\mathbb{E}\left[\big|f_1( \tilde x_{s(\Delta)}^{\varepsilon}, \tilde {y}_s^{\varepsilon})-f_1( \tilde x_{s(\Delta)}^{\varepsilon}, \hat{y}_s^{\varepsilon})\big|^2\right]d s\Big\}\cr
&\leq & C \int_0^T\mathbb{E}\left[\big|\tilde x_s^{\varepsilon}-\tilde x_{s(\Delta)}^{\varepsilon}\big|^2+\big|\tilde y_s^{\varepsilon}-\hat y_s^{\varepsilon}\big|^2\right]d s \cr
&\leq &C \Delta.
\end{eqnarray}
Here, the third inequality comes from the condition   $2\alpha<1$, while the final inequality comes from  \lemref{lem2}  and the estimate
\begin{eqnarray*}\label{3.42}
	\int_{0}^{T}\mathbb{E}\big[|\tilde y_{t}^{\varepsilon}-\hat y_{t}^{\varepsilon}|^2\big]dt \leq C \Delta,
\end{eqnarray*}
which was already shown in \cite[Lemma 4.4]{2020Large}.

Next we will estimate $J_2$. By using \cite[Lemma 7.1]{2002Rascanu}
and   (\ref{3-43}),  we have
\begin{eqnarray}\label{3.45}
J_2 &\le &C \varepsilon h^2(\varepsilon)\mathbb{E}\left[\sup _{t \in[0,T\wedge \tau_R]}\left|\int_0^t e^{-\lambda t}\left[(t-r)^{-2 \alpha}+r^{-\alpha}\right]\big\|\sigma_1( \tilde {x}^{\varepsilon})-\sigma_1( \hat{x}^{\varepsilon})\big\|_{\alpha,[0,r]} d r\right|^2\right] \cr
&\leq & C \varepsilon h^2(\varepsilon)\mathbb{E}\left[\sup _{t \in[0,T\wedge \tau_R]} \left| \int_0^t e^{-\lambda t}\left[(t-r)^{-2 \alpha}+r^{-\alpha}\right]\right.\right.\cr
&&\left.\times\left.\big(1+\Delta(\tilde{x}_r^{\varepsilon})+\Delta(\hat{x}_r^{\varepsilon})\big)\big\|\tilde {x}^{\varepsilon}-\hat{x}^{\varepsilon}\big\|_{\alpha,[0,r]}dr \right|^2\right],
\end{eqnarray}
where $\Delta(\tilde{x}_r^{\varepsilon})=\int_0^r \frac{|\tilde{x}_r^{\varepsilon}-\tilde{x}_q^{\varepsilon}|}{(r-q)^{1+\alpha}} d q$ and $\Delta(\hat{x}_r^{\varepsilon})=\int_0^r \frac{|\hat{x}_r^{\varepsilon}-\hat{x}_q^{\varepsilon}|}{(r-q)^{1+\alpha}} d q$, which are dominated by  a constant multiple of $\|\tilde x^{\varepsilon}\|_{\alpha, \infty}\le C$ and a constant multiple of $\|\hat x^{\varepsilon}\|_{\alpha, \infty}\le C$, where $C:=C(R,N,T)>0$ is independent of $\varepsilon$ and $\Delta$. 
Then,  we can easily see from  (\ref{3-3}) and (\ref{3-10}) that
\begin{eqnarray}\label{3.47}
J_2 &\le &C \varepsilon h^2(\varepsilon)\lambda^{2 \alpha-1} \mathbb{E}\big[\sup _{t \in[0,T\wedge \tau_R]} e^{-\lambda t}\|\tilde{x}^{\varepsilon}-\hat{x}^{\varepsilon}\|_{\alpha,[0,t]}^2 \big].
\end{eqnarray}

For the third term $J_3$, by invoking the same argument   as in \eqref{3.45},  we have
\begin{eqnarray}\label{3.48}
J_3 &\le &C \varepsilon \lambda^{2 \alpha-1} \mathbb{E}\big[\sup _{t \in[0,T\wedge \tau_R]} e^{-\lambda t}\|\tilde{x}^{\varepsilon}-\hat{x}^{\varepsilon}\|_{\alpha,[0,t]}^2 \big]
\end{eqnarray}
for some constant $C:=C(R,N,T)>0$ independent of $\varepsilon$ and $\Delta$.

Combining  (\ref{3.41})--(\ref{3.48}) and the condition that $\varepsilon,\varepsilon h^2(\varepsilon)\in(0,1]$,  {for large enough $\lambda$}, we obtain
\begin{eqnarray*}\label{3.49}
	\mathbb{E}\left[\sup _{t \in[0,T\wedge \tau_R]} e^{-\lambda t}\|\tilde x^{\varepsilon}-\hat{x}^{\varepsilon}\|_{\alpha,[0,t]}^2 \right]\le C\Delta ,
\end{eqnarray*}
which immediately implies that
\begin{eqnarray}\label{3.52}
	\mathbb{E}\left[\sup _{t \in[0,T\wedge \tau_R]} \|\tilde x^{\varepsilon}-\hat{x}^{\varepsilon}\|_{\alpha,[0,t]}^2 \right]\le C\Delta ,
\end{eqnarray}
where $C:=C(R,N,T)>0$ is independent of $\varepsilon$ and $\Delta$.

Next, we  will show the following inequality
\begin{eqnarray*}
\mathbb{E}\big[\sup _{t \in[0,T\wedge \tau_R]}\|\hat x^{\varepsilon}-\bar{x}^{\varepsilon}\|_{\alpha,[0,t]}^2\big] \leq C\left(\varepsilon \Delta^{-1}+\Delta\right),
\end{eqnarray*}
where $C:=C(R,N,T)>0$ is independent of $\varepsilon$ and $\Delta$. By straightforward computation, we have
{\begin{eqnarray*}\label{3.53}
		&&\mathbb{E}\big[\sup _{t \in[0,T\wedge \tau_R]}e^{-\lambda t}\|\hat x^{\varepsilon}-\bar{x}^{\varepsilon}\|_{\alpha,[0,t]}^2\big]\cr  &\leq & C \mathbb{E}\left[\sup _{t \in[0, T\wedge \tau_R]}e^{-\lambda t}\bigg\|\int_0^\cdot\big(f_1( \tilde x_{s(\Delta)}^{\varepsilon}, \hat{y}_s^{\varepsilon})-\bar f_1( \tilde x_{s(\Delta)}^{\varepsilon})\big) d s\bigg\|_{\alpha,[0,t]}^2 \right]\cr
		&&+C \mathbb{E}\left[\sup _{t \in[0,T\wedge \tau_R]}e^{-\lambda t}\bigg\|\int_0^\cdot\big(\bar f_1( \tilde x_{s(\Delta)}^{\varepsilon})-\bar f_1(  {\bar x_{s}^{\varepsilon}})\big) d s\bigg\|_{\alpha,[0,t]}^2 \right]\cr
		&&+C \mathbb{E}\left[\sup _{t \in[0,T\wedge \tau_R]}e^{-\lambda t}\bigg\|\int_0^\cdot\sqrt{\varepsilon} h(\varepsilon)\big(\sigma_1( \hat x_s^{\varepsilon})-\sigma_1( \bar x_{s}^{\varepsilon})\big) d u_s^\varepsilon\bigg\|_{\alpha,[0,t]}^2 \right] \cr
		&&+C \mathbb{E}\left[\sup _{t \in[0,T\wedge \tau_R]}e^{-\lambda t}\bigg\|\int_0^\cdot\sqrt{\varepsilon}\big(\sigma_1( \hat x_s^{\varepsilon})-\sigma_1( \bar x_{s}^{\varepsilon})\big) d B_s^H\bigg\|_{\alpha,[0,t]}^2 \right] \cr
		&=:& I_1+I_2+I_3+I_4.
\end{eqnarray*}}

Then, we compute the first term $I_1$,
\begin{eqnarray}\label{3.54}
I_1 &\leq & C \mathbb{E}\bigg[\sup _{t \in[0, T]}\bigg|\sum_{k=0}^{\left\lfloor\frac{t}{\Delta}\right\rfloor-1} \int_{k \Delta}^{(k+1) \Delta}\big(f_1(\tilde x_{k \Delta}^{\varepsilon}, \hat{y}_s^{\varepsilon})-\bar{f}_1( \tilde x_{k \Delta}^{\varepsilon})\big) d s\bigg|^2\bigg] \cr
&&+C \mathbb{E}\bigg[\sup _{t \in[0, T]}\bigg|\int_{\left\lfloor\frac{t}{\Delta}\right\rfloor \Delta}^t\big(f_1( \tilde x_{s(\Delta)}^{\varepsilon}, \hat{y}_s^{\varepsilon})-\bar{f}_1( \tilde x_{s(\Delta)}^{\varepsilon})\big) d s\bigg|^2\bigg] \cr
&&+C \mathbb{E}\bigg[\sup _{t \in[0, T]}\bigg(\int_0^t \frac{\left|\int_s^t\big(f_1( \tilde x_{r(\Delta)}^{\varepsilon}, \hat{y}_r^{\varepsilon})-\bar{f}_1( \tilde x_{r(\Delta)}^{\varepsilon}\big)) d r\right|}{(t-s)^{1+\alpha}} d s\bigg)^2\bigg] =:  \sum_{i=1}^3 {I}_1^i .
\end{eqnarray}
For the first two  terms $\sum_{i=1}^2 {I}_1^i$, we  obtain
\begin{eqnarray}\label{3.55}
\sum_{i=1}^2 {I}_1^i &\leq & C \mathbb{E}\big[\sup _{t \in[0, T]}\left\lfloor\frac{t}{\Delta}\right\rfloor \sum_{k=0}^{\left\lfloor\frac{t}{\Delta}\right\rfloor-1}\big|\int_{k \Delta}^{(k+1) \Delta}\big(f_1(\tilde x_{k \Delta}^{\varepsilon}, \hat{y}_s^{\varepsilon})-\bar{f}_1( \tilde x_{k \Delta}^{\varepsilon})\big) d s\big|^2\big] +C\Delta^2 \cr
&\leq & \frac{C}{\Delta^2} \max _{0 \leq k \leq\left\lfloor\frac{T}{\Delta}\right\rfloor-1} \mathbb{E}\big[\big|\int_{k \Delta}^{(k+1) \Delta}\big(f_1( \tilde x_{k \Delta}^{\varepsilon},\hat{y}_s^{\varepsilon})-\bar{f}_1( \tilde x_{k \Delta}^{\varepsilon})\big) d s\big|^2\big] +C \Delta^2 \cr
&\leq & C \frac{\varepsilon^2}{\Delta^2} \max _{0 \leq k \leq\left\lfloor\frac{T}{\Delta}\right\rfloor-1} \int_0^{\frac{\Delta}{\varepsilon}} \int_\zeta^{\frac{\Delta}{\varepsilon}} \mathcal{J}_k(s, \zeta) d s d \zeta+C \Delta^2,
\end{eqnarray}
where $0 \leq \zeta \leq s \leq \frac{\Delta}{\varepsilon}$ and
\begin{eqnarray}\label{3.56}
\mathcal{J}_k(s, \zeta)&=&\mathbb{E}\big[\big\langle f_1( \tilde x_{k \Delta}^{\varepsilon}, \hat{y}_{s \varepsilon+k \Delta}^{\varepsilon})-\bar{f}_1( \tilde x_{k \Delta}^{\varepsilon}),f_1( \tilde x_{k \Delta}^{\varepsilon}, \hat{y}_{\zeta \varepsilon+k \Delta}^{\varepsilon})-\bar{f}_1( \tilde x_{k \Delta}^{\varepsilon})\big\rangle\big].
\end{eqnarray}
It is known that \begin{eqnarray}\label{er}
\mathcal{J}_k(s, \zeta) \leq C e^{-\frac{\beta_1}{2}(s-\zeta)},
\end{eqnarray}
where $\beta_1$ is in (\textbf{A3'}). For a proof, see \cite[Appendix B]{2023Pei} for example.

To compute $I_1^3$, we first set $\ell:=\left\{t-s<2 \Delta\right\}$ and $\ell^c:=\left\{t-s \geq2 \Delta\right\}$.  By the condition that $2\alpha<1$, we have
\begin{eqnarray}\label{3.57}
I_1^3 
&\leq & C \mathbb{E}\big[\sup _{t \in[0, T]}\int_0^t{(t-s)^{-\frac{1}{2}-\alpha}}ds\times\sup _{t \in[0, T]}\int_0^t \frac{\big|\int_s^t\big(f_1( \tilde x_{r(\Delta)}^{\varepsilon}, \hat{y}_r^{\varepsilon})-\bar{f}_1( \tilde x_{r(\Delta)}^{\varepsilon}\big)) d r\big|^2}{(t-s)^{\frac{3}{2}+\alpha}} d s\big] \cr
&\leq & C \mathbb{E} \big[\sup _{t \in[0, T]}\int_0^t \frac{\big|\int_s^t\big(f_1( \tilde x_{r(\Delta)}^{\varepsilon}, \hat{y}_r^{\varepsilon})-\bar{f}_1( \tilde x_{r(\Delta)}^{\varepsilon}\big)) d r\big|^2}{(t-s)^{\frac{3}{2}+\alpha}}\mathbf{1}_{\ell^c} d s\big] \cr
&&+ C  \mathbb{E}\big[\sup _{t \in[0, T]}\int_0^t \frac{\big|\int_s^t\big(f_1( \tilde x_{r(\Delta)}^{\varepsilon}, \hat{y}_r^{\varepsilon})-\bar{f}_1( \tilde x_{r(\Delta)}^{\varepsilon}\big)) d r\big|^2}{(t-s)^{\frac{3}{2}+\alpha}} \mathbf{1}_\ell d s\big] \cr
&=:& {I}_1^{31}+ {I}_1^{32}.
\end{eqnarray}
Then, by (\textbf{A2'}), we estimate the first term ${I}_1^{31}$  as follows:
{\begin{eqnarray}\label{3.58}
{I}_1^{31} 
&\leq &C \mathbb{E} \bigg[\sup _{t \in[0, T]}\int_0^t \frac{\big|\int_s^t\big(f_1( \tilde x_{r(\Delta)}^{\varepsilon}, \hat{y}_r^{\varepsilon})-\bar{f}_1( \tilde x_{r(\Delta)}^{\varepsilon}\big)) d r\big|^2}{(t-s)^{\frac{3}{2}+\alpha}}\mathbf{1}_{\ell^c} d s\bigg] \cr
&\le&C \mathbb{E}\bigg[\sup _{t \in[0, T]} \int_0^t \frac{\left|\int_s^{\left\lfloor\frac{s}{\Delta}\right\rfloor \Delta+1}\big(f_1( \tilde x_{r(\Delta)}^{\varepsilon}, \hat{y}_r^{\varepsilon})-\bar{f}_1( \tilde x_{r(\Delta)}^{\varepsilon})\big) d r\right|^2}{(t-s)^{\frac{3}{2}+\alpha}} \mathbf{1}_{\ell^c} d s\bigg] \cr
&&+C \mathbb{E}\bigg[\sup _{t \in[0, T]} \int_0^t \frac{\left|\int_{\left\lfloor\frac{t}{\Delta}\right\rfloor \Delta}^t\big(f_1( \tilde x_{r(\Delta)}^{\varepsilon}, \hat{y}_r^{\varepsilon})-\bar{f}_1( \tilde x_{r(\Delta)}^{\varepsilon})\big) d r\right|^2}{(t-s)^{\frac{3}{2}+\alpha}} \mathbf{1}_{\ell^c} d s\bigg] \cr
&&+C \mathbb{E}\bigg[\sup _{t \in[0, T]} \int_0^t \frac{\left(\left\lfloor\frac{t}{\Delta}\right\rfloor-\left\lfloor\frac{s}{\Delta}\right\rfloor-1\right)}{(t-s)^{\frac{3}{2}+\alpha}}\bigg.\cr
&&\bigg.\qquad\qquad\qquad\times \sum_{k=\left\lfloor\frac{s}{\Delta}\right\rfloor+1}^{\left\lfloor\frac{t}{\Delta}\right\rfloor-1}\bigg|\int_{k \Delta}^{(k+1) \Delta}\big(f_1(\tilde x_{k \Delta}^{\varepsilon}, \hat{y}_r^{\varepsilon})-\bar{f}_1( \tilde x_{k \Delta}^{\varepsilon})\big) d r\bigg|^2 \mathbf{1}_{\ell^c} d s\bigg] \cr
&\leq& C \sup _{t \in[0, T]}\bigg(\int_0^t(t-s)^{-\frac{1}{2}-\alpha}((\left\lfloor\frac{s}{\Delta}\right\rfloor+1) \Delta-s) d s\bigg)\cr&&+C \sup _{t \in[0, T]}\bigg(\int_0^t(t-s)^{-\frac{1}{2}-\alpha}(t-\left\lfloor\frac{t}{\Delta}\right\rfloor \Delta)  d s\bigg) \cr
&&+C \Delta^{-1} \mathbb{E}\bigg[\sup _{t \in[0, T]} \int_0^t(t-s)^{-\frac{1}{2}-\alpha}\bigg.\cr
&&\bigg.\qquad\qquad\qquad\times \sum_{k=\left\lfloor\frac{s}{\Delta}\right\rfloor+1}^{\left\lfloor\frac{t}{\Delta}\right\rfloor-1}\big|\int_{k \Delta}^{(k+1) \Delta}\big(f_1( \tilde x_{k \Delta}^{\varepsilon}, \hat{y}_r^{\varepsilon})-\bar{f}_1( \tilde x_{k \Delta}^{\varepsilon})\big) d r\big|^2 \mathbf{1}_{\ell^c} d s\bigg] \cr
&\leq& C \Delta+C \frac{\varepsilon^2}{\Delta^2} \max _{0 \leq k \leq\left\lfloor\frac{T}{\Delta}\right\rfloor-1} \int_0^{\frac{\Delta}{\varepsilon}} \int_\zeta^{\frac{\Delta}{\varepsilon}} \mathcal{J}_k(s, \zeta) d s d \zeta.
\end{eqnarray}}
For the last inequality, we used the Schwarz inequality and a similar computation to \eqref{3.55}.

 {By  Assumption (\textbf{A2'}), \lemref{lem1} and the fact that $t-s< 2 \Delta$, we have}
\begin{eqnarray}\label{3.59}
{I}_1^{32}
&\leq & C \mathbb{E}\left[\sup _{t \in[\Delta, T]} \int_0^{t(\Delta)-\Delta} \frac{\big|\int_s^t\big(f_1(\tilde x_{r(\Delta)}^{\varepsilon}, \hat{y}_r^{\varepsilon})-\bar{f}_1(\tilde x_{r(\Delta)}^{\varepsilon})\big) d r\big|^2}{(t-s)^{\frac{3}{2}+\alpha}} \mathbf{1}_{ \ell} d s\right] \cr
&&+C \mathbb{E}\left[\sup _{t \in[\Delta, T]} \int_{t(\Delta)-\Delta}^t \frac{\big|\int_s^t\big(f_1(\tilde x_{r(\Delta)}^{\varepsilon}, \hat{y}_r^{\varepsilon})-\bar{f}_1(\tilde x_{r(\Delta)}^{\varepsilon})\big) d r\big|^2}{(t-s)^{\frac{3}{2}+\alpha}}\mathbf{1}_{ \ell} d s\right] \cr
&&+C\mathbb{E}\left[\sup _{t \in[0, \Delta]} \int_0^t \frac{\left|\int_s^t\big(f_1(\tilde x_{r(\Delta)}^{\varepsilon}, \hat{y}_r^{\varepsilon})-\bar{f}_1(\tilde x_{r(\Delta)}^{\varepsilon})\big) d r\right|^2}{(t-s)^{\frac{3}{2}+\alpha}} \mathbf{1}_{\ell} d s\right] \cr
&\leq & C \Delta \sup _{t \in[\Delta, T]}\big(\int_0^{t(\Delta)-\Delta}(t-s)^{-\frac{1}{2}-\alpha}  ds\big) +C \sup _{t \in[\Delta, T]}\big(\int_{t(\Delta)-\Delta}^t(t-s)^{\frac{1}{2}-\alpha} \mathbf{1}_{ \ell} ds\big) \cr
&&+C \sup _{t \in[0, \Delta]}\big(\int_0^t(t-s)^{\frac{1}{2}-\alpha} \mathbf{1}_{ \ell}  ds\big)\cr
&\leq & C \Delta.
\end{eqnarray}
Combining    (\ref{3.55})--(\ref{3.59}) and using \eqref{er}, we obtain
\begin{eqnarray}\label{3.60}
I_1
&\leq & C \Delta+C \frac{\varepsilon^2}{\Delta^2} \max _{0 \leq k \leq\left\lfloor\frac{T}{\Delta}\right\rfloor-1} \int_0^{\frac{\Delta}{\varepsilon}} \int_\zeta^{\frac{\Delta}{\varepsilon}} \mathcal{J}_k(s, \zeta) d s d \zeta\cr
& \leq& C \frac{\varepsilon^2}{\Delta^2} \max _{0 \leq k \leq\left\lfloor\frac{T}{\Delta}\right\rfloor-1} \int_0^{\frac{\Delta}{\varepsilon}} \int_\zeta^{\frac{\Delta}{\varepsilon}} e^{-\frac{\beta_1}{2}(s-\zeta)} d s d \zeta+C \Delta \cr
& \leq &C \frac{\varepsilon^2}{\Delta^2}\left(\frac{2}{\beta_1} \frac{\Delta}{\varepsilon}-\frac{4}{\beta_1^2}+e^{\frac{-\beta_1}{2} \frac{\Delta}{\varepsilon}}\right)+C\Delta \cr
& \leq& C\left(\varepsilon \Delta^{-1}+\Delta\right).
\end{eqnarray}

Afterwards, it goes on to estimate $I_2$,
{\begin{eqnarray}\label{3.61}
I_2
&\leq &  C \mathbb{E}\bigg[\sup _{t \in[0,T\wedge \tau_R]} e^{-\lambda t}\big\|\int_0^\cdot\big(\bar{f}_1( \tilde {x}_{s(\Delta)}^{\varepsilon})-\bar{f}_1( \tilde {x}_s^{\varepsilon})\big) d s\big\|_{\alpha,[0,t]}^2 \bigg] \cr
&&+C \mathbb{E}\bigg[\sup _{t \in[0,T\wedge \tau_R]} e^{-\lambda t}\big\|\int_0^\cdot\big(\bar{f}_1(\tilde {x}_s^{\varepsilon})-\bar{f}_1(\hat {x}_s^{\varepsilon})\big) d s\big\|_{\alpha,[0,t]}^2 \bigg] \cr
&&+C \mathbb{E}\bigg[\sup _{t \in[0,T\wedge \tau_R]} e^{-\lambda t}\big\|\int_0^\cdot\big(\bar{f}_1( \hat {x}_s^{\varepsilon})-\bar{f}_1( \bar {x}_s^{\varepsilon})\big) d s\big\|_{\alpha,[0,t]}^2 \bigg]\cr&=:& I_2^1+I_2^2+I_2^3.
\end{eqnarray}}
By taking similar estimates as in (\ref{3.41}) and using   (\textbf{A2'}) and  \lemref{lem1}, we have
\begin{eqnarray}\label{3.62}
\sum_{i=1}^{2} I_2^i 
&\leq & C \int_0^{T\wedge \tau_R}\mathbb{E}\left[\big|\bar{f}_1( \tilde x_s^{\varepsilon})-\bar{f}_1( \tilde x_{s(\Delta)}^{\varepsilon})\big|^2 \right]d s +C  {\int_0^T\mathbb{E}\left[\big|\bar{f}_1( \tilde x_{s}^{\varepsilon})-\bar{f}_1( \hat x_{s}^{\varepsilon})\big|^2\right]d s} \cr
&\leq & C \int_0^{T\wedge \tau_R}\mathbb{E}\left[\big|\tilde x_s^{\varepsilon}-\tilde x_{s(\Delta)}^{\varepsilon}\big|^2+\big|\hat x_s^{\varepsilon}-\tilde x_{s}^{\varepsilon}\big|^2\right]d s \cr
&\leq &C \Delta.
\end{eqnarray}
for some constants $C:=C(R,N,T)>0$ independent of $\varepsilon,\Delta$.
By the Lipschitz continuity of averaged coefficient $\bar f_1$, we obtain
\begin{eqnarray}\label{3.63}
I_2^3 
&\leq  & C \mathbb{E}\left[\sup _{t \in[0, T\wedge \tau_R]} e^{-\lambda t} \int_0^t(t-s)^{-2 \alpha} \big| \bar{f}_1( \hat{x}_s^{\varepsilon})-\bar{f}_1( \bar{x}_s^{\varepsilon})\big|^2  d s\right] \cr
& \leq& C\mathbb{E}\left[\sup _{t \in[0, T\wedge \tau_R]} \int_0^t e^{-\lambda(t-s)}(t-s)^{-2 \alpha} e^{-\lambda s}\big|\hat{x}_s^{\varepsilon}-\bar{x}_s^{\varepsilon}\big|^2  d s\right] \cr
& \leq& C\mathbb{E}\left[\sup _{t \in[0, T\wedge \tau_R]} e^{-\lambda t}\big\|\hat{x}^{\varepsilon}-\bar{x}^{\varepsilon}\big\|_{\alpha,[0,t]}^2 \right] \sup _{t \in[0, T]} \int_0^t e^{-\lambda(t-r)}(t-r)^{-2 \alpha} d r \cr
& \leq& C \lambda^{2 \alpha-1} \mathbb{E}\big[\sup_{t \in[0, T\wedge \tau_R]}e^{-\lambda t} \big\|\hat{x}^{\varepsilon}-\bar{x}^{\varepsilon}\big\|_{\alpha,[0,t]}^2 \big].
\end{eqnarray}
 {The $I_2$ is estimated by summarizing \eqref{3.61}--\eqref{3.63},
\begin{eqnarray}\label{3.69}
I_2 \le C\Delta+C \lambda^{2 \alpha-1} \mathbb{E}\big[\sup_{t \in[0, T\wedge \tau_R]}e^{-\lambda t} \big\|\hat{x}^{\varepsilon}-\bar{x}^{\varepsilon}\big\|_{\alpha,[0,t]}^2 \big].
\end{eqnarray}}
For the term $I_3$, by taking similar estimates as in   (\ref{3.45}), we have
\begin{eqnarray*}\label{3.64}
	I_3 
	&\leq & C \varepsilon h^2(\varepsilon)\mathbb{E}\left[\sup _{t \in[0, T\wedge \tau_R]}\left|\int_0^t e^{-\lambda t}\left[(t-r)^{-2 \alpha}+r^{-\alpha}\right]\left\|\sigma_1(  \hat {x}^{\varepsilon})-\sigma_1(  \bar{x}^{\varepsilon})\right\|_{\alpha,[0,s]} d r\right|^2\right] \cr
	&\leq & C\varepsilon h^2(\varepsilon)\mathbb{E}\bigg[\sup _{t \in[0, T\wedge \tau_R]} \mid \int_0^t e^{-\lambda t}\big[(t-r)^{-2 \alpha}+r^{-\alpha}\big]\bigg.\cr
	&&\bigg.\times\left.\left(1+\Delta(  \hat {x}_r^{\varepsilon})+\Delta( \bar {x}_r^{\varepsilon})\right)\| \hat {x}^{\varepsilon}- \bar {x}^{\varepsilon}\|_{\alpha,[0,r]}  d r\right|^2\bigg].
\end{eqnarray*}
Here,  the two terms $\Delta(\hat{x}_r^{\varepsilon})=\int_0^r \frac{|\hat{x}_r^{\varepsilon}-\hat{x}_q^{\varepsilon}|}{(r-q)^{1+\alpha}} d q$ and $\Delta(\bar{x}_r^{\varepsilon})=\int_0^r \frac{|\bar{x}_r^{\varepsilon}-\bar{x}_q^{\varepsilon}|}{(r-q)^{1+\alpha}} d q$  can be dominated by a constant multiple of $\|\hat x^{\varepsilon}\|_{\alpha, \infty}\le C$ and $\|\bar x^{\varepsilon}\|_{\alpha, \infty}\le C$, where   $C>0$ is independent of $\varepsilon$ and $\Delta$.
Hence, 
\begin{eqnarray}\label{3.66}
I_3 
& \leq& C \varepsilon h^2(\varepsilon)\lambda^{2 \alpha-1} \mathbb{E}\big[\sup_{t \in[0, T\wedge \tau_R]} \big\|\hat{x}_t^{\varepsilon}-\bar{x}_t^{\varepsilon}\big\|_{\alpha,[0,t]}^2 \big].
\end{eqnarray}
For the last term $I_4$, by take similar estimates as in (\ref{3.66}), we obtain
\begin{eqnarray}\label{3.68}
I_4 
\leq C \varepsilon \lambda^{2 \alpha-1} \mathbb{E}\big[\sup_{t \in[0, T\wedge \tau_R]} \big\|\hat{x}_s^{\varepsilon}-\bar{x}_s^{\varepsilon}\big\|_{\alpha,[0,t]}^2 \big].
\end{eqnarray}

Combining  (\ref{3.60})--(\ref{3.68}) and the condition that $\varepsilon,\varepsilon h^2(\varepsilon)\in(0,1]$, we obtain
\begin{eqnarray*}\label{3.67}
	\mathbb{E}\big[\sup _{t \in[0,T\wedge \tau_R]}\|\hat x^{\varepsilon}-\bar{x}^{\varepsilon}\|_{\alpha,[0,t]}^2\big] 
	& \leq& C \lambda^{2 \alpha-1} \mathbb{E}\big[\sup_{t \in[0, T\wedge \tau_R]} \big\|\hat{x}^{\varepsilon}-\bar{x}^{\varepsilon}\big\|_{\alpha,[0,t]}^2 \big]\cr
	&&+C\left(\varepsilon \Delta^{-1}+\Delta\right),
\end{eqnarray*}
where $C:=C(R,N,T)>0$ is independent of $\varepsilon$ and $\Delta$.
Taking $\lambda$ large enough such that $C  \lambda^{2 \alpha-1}<1$, we have
\begin{eqnarray}\label{3.80}
\mathbb{E}\big[\sup _{t \in[0,T\wedge \tau_R]}\|\hat x^{\varepsilon}-\bar{x}^{\varepsilon}\|_{\alpha,[0,t]}^2\big] 
\le C\left(\varepsilon \Delta^{-1}+\Delta\right)
\end{eqnarray}
for some constants $C:=C(R,N,T)>0$ independent of $\varepsilon$ and $\Delta$.

Then, by combining with estimate (\ref{3.52}) and (\ref{3.80}), we obtain 
\begin{eqnarray}\label{3.83}
\mathbb{E}\big[\sup _{t \in[0,T\wedge \tau_R]}\|\tilde x^{\varepsilon}-\bar{x}^{\varepsilon}\|_{\alpha,[0,t]}^2\big] 
&\le& C\left(\varepsilon \Delta^{-1}+\Delta\right).
\end{eqnarray}
Then, we  obtain  
\begin{eqnarray}\label{3.84}
\mathbb{E}\big[\sup _{t \in[0,T\wedge \tau_R]}\frac{1}{\varepsilon h^2(\varepsilon)}\|\tilde x^{\varepsilon}-\bar{x}^{\varepsilon}\|_{\alpha,[0,t]}^2\big] 
&\le&  C\left(\frac{\varepsilon}{\Delta\varepsilon h^2(\varepsilon)}+\frac{\Delta}{\varepsilon h^2(\varepsilon)}\right).
\end{eqnarray}
 {Precisely, we could take suitable $0<\Delta=\Delta(\varepsilon)<1$ such that  when $\varepsilon \to 0$
\begin{eqnarray*}
	\frac{\Delta}{\varepsilon h^2(\varepsilon)}	\to 0,\quad
	\frac{\varepsilon}{\Delta\varepsilon h^2(\varepsilon)}	\to 0.
\end{eqnarray*}
For example, when $h(\varepsilon)=\varepsilon^{-\frac{\theta}{2}}$ with $\theta\in(0,1)$, we could take $\Delta=\Delta(\varepsilon)=\varepsilon^{\gamma+1}h^2(\varepsilon)|\ln \varepsilon|$ with $\gamma\in(0,\theta-1/2)$.}

For each fixed large $R$, we see that
\begin{eqnarray*}\label{3.84'}
	\lim_{\varepsilon \to 0}\mathbb{E}\big[\sup _{t \in[0,T\wedge \tau_R]}\frac{1}{\varepsilon h^2(\varepsilon)}\|\tilde x^{\varepsilon}-\bar{x}^{\varepsilon}\|_{\alpha,[0,t]}^2\big]= 0.
\end{eqnarray*}
Then, we have 
\begin{eqnarray}\label{main}
&&\mathbb{P}\big(\sup _{t \in[0, T]}\frac{1}{\varepsilon h^2(\varepsilon)}\|\tilde x^{\varepsilon}-\bar{x}^{\varepsilon}\|_{\alpha,[0,t]}^2\geq \delta\big)\cr &\leqslant & \mathbb{P}\left(T>{\tau}_R\right)+\mathbb{P}\big(\sup _{t \in[0,T\wedge \tau_R]}\frac{1}{\varepsilon h^2(\varepsilon)}\|\tilde x^{\varepsilon}-\bar{x}^{\varepsilon}\|_{\alpha,[0,t]}^2 \geq \delta, T \leqslant {\tau}_R\big) \cr
&\leqslant & \mathbb{P}\big(\|B^H\|_{1-\alpha, \infty, T} \geq R\big)+\mathbb{P}\big(\sup _{t \in[0,T\wedge \tau_R]}\frac{1}{\varepsilon h^2(\varepsilon)}\|\tilde x^{\varepsilon}-\bar{x}^{\varepsilon}\|_{\alpha,[0,t]}^2 \geq \delta\big) \cr
&\leqslant & \mathbb{P}\big(\|B^H\|_{1-\alpha, \infty, T} \geq R\big)  +\frac{1}{\delta}\mathbb{E}\big[\sup _{t \in[0,T\wedge \tau_R]}\frac{1}{\varepsilon h^2(\varepsilon)}\|\tilde x^{\varepsilon}-\bar{x}^{\varepsilon}\|_{\alpha,[0,t]}^2  \big] ,
\end{eqnarray}
where the final inequality is from the Chebyshev inquality.

For each fixed  $R>0$ and $\delta>0$, we take $\limsup_{\varepsilon \rightarrow 0}$, the second term on the right hand side of \eqref{main} converges to $0$. Next let $R \to\infty$.  Then the first term on the right hand side of \eqref{main} converges to $0$. Hence,
\begin{eqnarray*}\label{3.85'}
	\limsup_{\varepsilon \to 0}\mathbb{P}\big(\sup _{t \in[0, T]}\frac{1}{\varepsilon h^2(\varepsilon)}\|\tilde x^{\varepsilon}-\bar{x}^{\varepsilon}\|_{\alpha,[0,t]}^2\geq \delta\big)= 0.
\end{eqnarray*}
Next, by the continuous inclusion $C^{\alpha+\kappa} \subset W_0^{\alpha, \infty} \subset C^{\alpha-\kappa}$ holds for any  small enough $\kappa>0$,  we have
\begin{eqnarray*}\label{3.81}
	\lim_{\varepsilon \to 0}\mathbb{P}\big(\sup _{t \in[0,T\wedge \tau_R]}\frac{1}{\varepsilon h^2(\varepsilon)}\|\tilde x^{\varepsilon}-\bar{x}^{\varepsilon}\|_{\alpha,[0,t]}^2\geq \delta\big)= 0.
\end{eqnarray*}
 {Due to the definition of $\hat z^{\varepsilon}$ given in \eqref{3.78},  that 
for any $\delta>0$,
\begin{eqnarray*}
	\lim_{\varepsilon \to 0}\mathbb{P}\big(\|\hat z^\varepsilon\|^2_{\alpha\textrm{-hld}}\geq\delta\big)=0.
\end{eqnarray*}}
By using Step 2 of \thmref{thm1}, we can prove that if  $(u^{\varepsilon}, v^{\varepsilon})\in \mathcal{A}^{N}_b$  converges in distribution to $(u, v)$ as $\varepsilon \to 0$,   $\bar z^{\varepsilon}=\mathcal{G}^0 (u^{\varepsilon}, v^{\varepsilon})$  converges in distribution to $\tilde z=\mathcal{G}^0 (u, v)$ in $ C^{\alpha}([0,T],\mathbb{R}^m)$ as $\varepsilon \to 0$. Then we show that the statement \textbf{(b)} holds.

\textbf{Step 3}.  {Similar to Step 3 of \thmref{thm1}}, the proof is finished. \qed

\section*{Acknowledgments}

 {The authors are grateful to Professor Wei Liu  for helpful comments about Example 4.1.}
This work was partly supported by the Key International (Regional) Cooperative Research Projects of the NSF of China (Grant 12120101002), the NSF of China (Grant 12072264)
and JSPS KAKENHI (Grant No. 20H01807).

\end{document}